\documentclass[BCOR=12mm, DIV=default,twoside, pagesize, draft]{scrartcl}

\addtokomafont{caption}{\small}
\setkomafont{captionlabel}{\sffamily\bfseries}
 
\usepackage{tikz}   
    
\usetikzlibrary{arrows,automata}  
     
\usepackage{amsmath, amssymb, amsthm} 
\usepackage[utf8x]{inputenc}
\usepackage{libertine}
\usepackage[T1]{fontenc}
\usepackage[english]{babel}
\usepackage{url}  
\usepackage{faktor,csquotes,color}
\usepackage{pgfplots} 
\usepackage{enumitem}

		\usepackage{float} 
		\usetikzlibrary{arrows}
		\usetikzlibrary{datavisualization.formats.functions}
		\pgfplotsset{compat=1.12}
		\usepgfplotslibrary{fillbetween}
		\usetikzlibrary{patterns}
		\usepackage{accents}
\newcommand{\ut}[1]{\underaccent{\tilde}{#1}}
\newcommand{\utt}[1]{\underaccent{\hat}{#1}}

\usepackage{changes}

\definechangesauthor[name={Soeren}, color=red]{soe}
\definechangesauthor[name={Kristoffer}, color=blue]{kri}
\newcommand{\stkout}[1]{\ifmmode\text{\sout{\ensuremath{#1}}}\else\sout{#1}\fi}
\setdeletedmarkup{\stkout{#1}}

\usepackage{units,varioref}

\usepackage[shortcuts]{extdash}
\usepackage{setspace,scrpage2}
\usepackage{mathtools}
\mathtoolsset{showonlyrefs,showmanualtags}
\numberwithin{equation}{section}

\allowdisplaybreaks 
\clearscrheadfoot
\automark[section]{section}
\ohead{\headmark}
\ofoot[\pagemark]{\pagemark}

\newtheoremstyle{break}{\topsep}{\topsep}{\itshape}{}{\bfseries}{.}{\newline}{}
\newtheoremstyle{exampl}{\topsep}{\topsep}{\upshape}{}{\bfseries}{.}{\newline}{}

\theoremstyle{plain}
\newtheorem{thm}{Theorem}[section]
\newtheorem{lem}[thm]{Lemma}  
\newtheorem{prop}[thm]{Proposition}  
\newtheorem{cor}[thm]{Corollary} 
\newtheorem{ass}[thm]{Assumption} 

\theoremstyle{break}

\theoremstyle{definition}
\newtheorem{defi}[thm]{Definition}

\theoremstyle{exampl}

\theoremstyle{remark}
\newtheorem{rem}[thm]{Remark}

\DeclareMathOperator{\E}{\mathbb E}

\definecolor{mygray}{gray}{.5}

\title{Moment constrained optimal dividends: precommitment \& consistent planning}
\author{S\"oren Christensen\footnote{Department of Mathematics, Kiel University, Germany. E-mail address: christensen@math.uni-kiel.de.} \and Kristoffer Lindensj\"o\footnote{Department of Mathematics, Stockholm University, Sweden. E-mail address: kristoffer.lindensjo@math.su.se.}}

\date{\today}

\begin{document}
\maketitle

\begin{abstract} A moment constraint that limits the number of dividends in the optimal dividend problem is suggested. This leads to a new type of time-inconsistent stochastic impulse control problem. 
First, the optimal solution in the precommitment sense is derived. 
Second, the problem is formulated as an intrapersonal sequential dynamic game in line with Strotz' consistent planning.  In particular, the notions of pure dividend strategies and a (strong) subgame perfect Nash equilibrium are adapted. An equilibrium is derived using a smooth fit condition. The equilibrium is shown to be strong. 
The uncontrolled state process is a fairly general diffusion. 
\end{abstract}

\vspace{1mm}
 

\section{Introduction} \label{sec:intro}
We consider a family of filtered probability spaces $(\Omega,\mathcal{F},(\mathcal{F}_t)_{t\geq0},\mathbb{P}_x),x\in{\mathbb{R}},$ satisfying the usual conditions and a one-dimensional process $X=(X_t)_{t\geq 0}$ given under $\mathbb{P}_x$ by 
 \begin{align} 
dX_t =  \mu(X_t) dt + \sigma(X_t) dW_t - dD_t, \enskip X_{0}=x \enskip \mbox{a.s.}\label{state-process}
\end{align}
where 
$D=(D_t)_{t\geq 0}$ is a non-decreasing adapted process and $W=(W_t)_{t\geq 0}$ is a Wiener process. The associated expectations are denoted by $\mathbb{E}_x$. The classical optimal dividend problem in this setting is to suppose that the owner of an insurance company with surplus process $X$ chooses the dividend policy $D$ that maximizes the sum of discounted dividend payments until bankruptcy. Specifically, the owner considers the stochastic control problem
\begin{align}  
\begin{split}
U(x)&:=\sup_{D\in \mathcal A(x)}J(x;D), \enskip J(x;D):= \mathbb{E}_x\left(\int_0^\tau e^{-rt}dD_t\right),\\ 
\tau&:=\inf\{t\geq 0: X_t\leq0\}, \label{classical-problem}
\end{split}
\end{align}
where $\tau$ is interpreted as the bankruptcy time, $r>0$ is a discount factor and
%
%
%
\begin{align}  
&\mbox{$D \in \mathcal A(x)$ if $D$ is a LCRL non-decreasing adapted process with $D_{0}=0$}\\
&\mbox{such that $X_{\tau+}\geq0$.}
\end{align}  
Problem \eqref{classical-problem} was first studied in \cite{shreve1984optimal} where --- under certain conditions for the functions $\mu(\cdot)$ and $\sigma(\cdot)$; notably, $\mu'(x)\leq r$ for all $x\geq 0$ --- it was found that if an optimal policy exists then it is to pay dividends only in order to reflect the process $X$ at a barrier $x^*$ and if no optimal policy exists then the optimal value function $U(x)$ is the limit of the value function given a reflecting barrier dividend policy when sending the barrier to infinity. 
A criticism of this formulation of the dividend problem  from an economic viewpoint is that the solution involves an unreasonably high number of dividend payments; in particular, once $X$ reaches the barrier $x^*$ an infinite number of dividends will be paid during any immediately following time interval no matter how small.  
One way of taking this criticism into account is to introduce a fixed cost for each dividend payment, which obviously limits the optimal number of dividend payments and thus leads to a stochastic impulse control problem.  
In the present paper we instead introduce a moment constraint that more directly limits the number of dividend payments. Specifically, if we denote by $\tau_n$ the timing of the $n$:th dividend payment for a discrete dividend policy $D$, then $D$ is in the present paper said to be admissible for a given initial surplus $x\geq0$ in \eqref{state-process}, which we write as $D \in \mathcal A(x,k)$, if $D \in \mathcal A(x)$ and the moment constraint $\mathbb{E}_{x}\left( \sum_{n:\tau_n\leq \tau}e^{-r\tau_n}\right) \leq \frac{1}{k}$ is satisfied, where $k>0$ is a fixed parameter. 
%
%
Clearly, a dividend policy satisfying the moment constraint  must be of impulse control type; in particular, an admissible dividend policy $D\in \mathcal A(x,k)$ can be represented as 
%
%
\begin{align}
\begin{split}  
&\mbox{$D_t =\sum_{n:\tau_n<t}\zeta_n, t\geq 0$ with $D_0=0$, where, for each $n=1,2...$,}\\
&\mbox{$\tau_n\geq0$ is an $(\mathcal{F}_t)_{t\geq 0}$-stopping time such that $\tau_{n+1}>\tau_n$ on $\{\tau_n<\infty\}$}\\
&\mbox{and $\zeta_n$ is an $\mathcal{F}_{\tau_n}$-measurable random variable with $0<\zeta_n\leq X_{\tau_n}$ a.s.,}\\
&\mbox{and $\tau_n \rightarrow \infty$ a.s. $n \rightarrow \infty$.}\label{condaaa}\\
\end{split}
\end{align}
Note also that $\zeta_n = X_{\tau_n}-X_{\tau_n+}$. 
The interpretation of $\zeta_n$ is that it is the $n$:th dividend payment. In the sequel we denote a dividend policy of impulse control type as defined in \eqref{condaaa} by $S=(\tau_n,\zeta_n)_{n\geq 1}$ --- which means, using a slight abuse of notation, that $S=(\tau_n,\zeta_n)_{n\geq 1} \in \mathcal A(x)$ for each $x\geq0$ by definition --- and for ease of exposition we restate the moment constraint as
%
%
\begin{align}  
R(x;S):=\mathbb{E}_{x}\left( \sum_{n:\tau_n\leq \tau}e^{-r\tau_n}\right) \leq \frac{1}{k}, \enskip\mbox{where $k>0$ is fixed}.
\enskip \tag{MC}\label{constraint2}
\end{align}
This implies that 
\begin{align}  
&\mbox{$S=(\tau_n,\zeta_n)_{n\geq 1}\in \mathcal A(x;k)$ if \eqref{constraint2} holds.}
\end{align}
The objective of the present paper is to study the problem of maximizing the sum of expected discounted dividends over the set of admissible dividend policies $\mathcal A(x;k)$, i.e. in particular under the constraint \eqref{constraint2}. It turns out that this problem is time-inconsistent in the sense that a dividend policy which is, in the precommitment sense, optimal at time $0$ will not generally be optimal at a time $t>0$  when considering the constraint \eqref{constraint2} using the value for the state process at $t$, 
see Remark \ref{pre-commitment-prob-is-time-incons}. 
We remark that it could be argued that it would be more reasonable to call this problem space-inconsistent but we have chosen to use the more established term. 
The main contribution of the present paper is to formulate and solve this problem both in the precommitment sense and in the game-theoretic sense of Strotz' consistent planning. The present paper is, to our knowledge, the first to study a stochastic control problem that is time-inconsistent due to a constraint 
using the game-theoretic approach and we note that it seems likely that there are many other interesting constrained stochastic control problems that can be formulated and solved along the lines of the present paper.

An interpretation of the constraint \eqref{constraint2} is that the company wants to limit the number of dividend payments not mainly due to financial costs of paying dividends but rather because a high number of dividend payments is undesirable for other reasons; for example because the financial market expects a limited number of dividend payments, or because they involve tedious administrative work for the decision maker. Note that $\frac{1}{k}$ can be interpreted as the maximum number of expected dividends to be paid until an independent exponential time with expectation $\frac{1}{r}$.

The rest of the paper is structured as follows. Section \ref{sec:prev-lit} mentions related literature. In Section \ref{sec:problem} the model of the present paper is formulated in more detail and some results that will be used in the sequel are presented. 
In Section \ref{sec:pre-commitment-solution} the precommitment interpretation of the constrained dividend problem is formulated and solved; 
and properties of the solution are investigated. 
These results rely on the solution to the (unconstrained) optimal dividend problem under the assumption that a fixed cost is incurred for each dividend payment, which is therefore also recapitulated and studied in Section \ref{sec:pre-commitment-solution}. 
In Section \ref{sec:time-consistent-solution} we formulate and solve the constrained dividend problem as a game along the lines of Strotz' consistent planning. 
An equilibrium is derived using a smooth fit condition. 
 The equilibrium is shown to be strong  
and its properties are investigated. 
%
A discussion of our equilibrium definition is found in Section \ref{discussion-eq}. 
An example is studied in Section \ref{example:sec}. 
Model assumptions are discussed in Appendix \ref{app-model}. 
Most proofs are found in Appendix \ref{app-proofs}.

\subsection{Background and related literature} \label{sec:prev-lit}
Th study of time-inconsistent control problems goes back to a seminal paper by Strotz \cite{strotz} in the 1950s, but the field has experienced a considerable activity during the last years. Time-inconsistency in stochastic control typically arises due to the consideration of 
(1) non-exponential discounting,
(2) a state-dependent reward function, or 
(3) nonlinearities in the expected reward, e.g. mean-variance utility; see e.g. \cite{tomas-disc,christensen2017finding,christensen2018time,lindensjo2017timeinconHJB} for descriptions of these kinds of problems and references.   
Time-inconsistency is typically studied using the precommitment approach, which means finding an optimal control policy for a given initial value of the controlled process, or the time-consistent (game-theoretic) approach along the lines of Strotz' invention. Time-inconsistency can also be studied using the notion of dynamic optimality, see \cite{pedersen2016optimal,pedersen2013optimal}. 

The game-theoretic approach is to interpret a time-inconsistent problem as an intrapersonal sequential dynamic game. The approach is formalized by defining a subgame perfect Nash equilibrium suitable for the particular problem at hand. See Section \ref{discussion-eq} for an interpretation of the game in the present paper and Section \ref{sec:time-consistent-solution} for our equilibrium definition. A main reference for the general theory of the game-theoretic approach to time-inconsistent stochastic control is \cite{bjork2017time}. 
A large literature studying particular time-inconsistent problems using the game-theoretic approach has evolved during the last years; a short recent survey is contained in \cite{lindensjo2017timeinconHJB}. The general theory of time-inconsistent stopping is studied in e.g. \cite{christensen2017finding,christensen2018time,huang2018time}. 

 The present paper is different from most papers on time-inconsistent control in the sense that the time-inconsistency does not arise due to the factors (1)--(3) mentioned above; instead it is due to the consideration of a constraint for an otherwise time-consistent stochastic control problem. While we believe, as mentioned above, that the present paper is the first to consider the game-theoretic approach to a problem that is time-inconsistent for this reason there are many papers that study stochastic control under different kinds of constraints and we here only mention a few. 
Optimal dividends under ruin probability constraints are studied in \cite{dickson2006optimal,hipp2003optimal}, while optimal dividends under a constraint for the ruin time is studied in \cite{hernandez2018time}. A dividend problem under the constraint that the surplus process must be above a given fixed level in order for dividend payments to be admissible is studied in \cite{paulsen2003optimal}; see also \cite{lindensjo2019optimal} where this problem is studied in a model which allows for capital injection. 
Optimal stopping under expectation constraints is studied in \cite{ankirchner2019verification,bayraktar2017dynamic} while stochastic control under expectation constraints is studied in \cite{yu2018dynamic}. Distribution-constrained optimal stopping is studied in  
\cite{
bayraktar2019distribution,
beiglbock2016geometry}. 
It should also be mentioned that mean-variance problems are sometimes formulated as constrained optimization  problems. For example,  constrained mean-variance portfolio selection (a control problem) is studied in \cite{pedersen2013optimal} and
constrained mean-variance selling strategies (a stopping problem) are studied in \cite{pedersen2016optimal}, although the main topic of these papers is the notion of dynamic optimality. In \cite{pedersen2018constrained} a constrained portfolio selection problem is investigated using the dynamic optimality approach and a comparison is made to the precommitment approach. 
%
%
%
%
The game-theoretic approach to a mean-variance optimization problem under the constraint of no short selling is studied in \cite{Bensoussan2018}. We also mention \cite{nutz2019conditional} in which a conditional optimal stopping problem is studied using a game-theoretic approach.

Time-inconsistent dividend problems have been studied before: The optimal dividend problem under non-exponential discounting is studied using the game-theoretic approach in \cite{chen2014optimal,chen2018optimal,chen2017optimal,li2016equilibrium,zhao2014dividend,zhu2019singular}, while \cite{chen2016optimal} studies this problem incorporating also capital injections. 
%
%
%
%
%
%
%

The precommitment approach of the present paper relies, as we have mentioned, on results for the fixed cost dividend problem. This problem was first studied in \cite{jeanblanc1995optimization} which considers a Wiener process with drift and later in \cite{paulsen2007optimal} where a more general diffusion model is considered; see Remarks \ref{sol-general-fixed-cost-problem} and \ref{rem-about-fixedcostres} and Appendix \ref{app-model} for further references. 

There is a vast literature on many different versions of the optimal dividend problem, see e.g. the literature reviews \cite{albrecher2009optimality,avanzi2009strategies} and the more recent surveys included in \cite{de2017dividend,optimaldividendFerrari2019,lindensjo2019optimal}.


\section{Model formulation and preliminaries} \label{sec:problem}

In this section we specify model assumptions and present results and notation on which the subsequent analysis relies. 
%
Unless otherwise stated we assume throughout the paper that all items in Assumption \ref{coeff-assum} below hold;  see Appendix \ref{app-model} for a discussion of Assumption \ref{coeff-assum}.

\begin{ass} \label{coeff-assum} \quad 
\begin{enumerate}[label=(A.\arabic*)] 


\item  \label{coeff-assum:2} $\mu(\cdot)$ and $\sigma(\cdot)$ are continuously differentiable and Lipschitz continous
, and $\mu'(\cdot)$ and $\sigma'(\cdot)$ are Lipschitz continous.

\item \label{coeff-assum:3} $\sigma^2(x)>0$ for all $x\geq0$.

\item \label{coeff-assum:4} $\mu'(x)<r$ for all $x\geq0$ (recall that $r>0$ is the discount rate).

\item \label{coeff-assum:5} An $\varepsilon>0$ and an $x_a\geq 0$ such that $\mu'(x)<r-\varepsilon$ for all $x\geq x_a$ exist.

\item \label{coeff-assum:6} $\mu(0) > 0$.
\end{enumerate}
\end{ass}
(Note that 
\ref{coeff-assum:4} and \ref{coeff-assum:5} are, given that \ref{coeff-assum:2} holds, equivalent to the condition that
there exists an $\varepsilon>0$ such that $\mu'(x)<r-\varepsilon$ for all $x\geq 0$.) 
Consider the boundary value problem 
\begin{align}
&A_X g(x):=\mu(x) g'(x) +\frac{1}{2} \sigma^2(x) g''(x)=r g(x), \enskip x>0 \label{ODE1}\\
& g'(0)>0,  \enskip g(0)=0, \enskip g(\cdot)\in C^2(0,\infty).\label{ODE2}
\end{align}

\begin{lem} \label{canon-sol-lemma} Suppose \ref{coeff-assum:2}--\ref{coeff-assum:3} hold. Then, a solution $g(\cdot)$ of \eqref{ODE1}--\eqref{ODE2}  that is unique up to multiplication of a positive constant  exists (and is in the sequel called a canonical solution). Moreover:
\begin{enumerate}[label=(\roman*)] 

\item \label{canon-sol-lemma:1} $g(\cdot)\in C^3(0,\infty)$ and $g'''(\cdot)$ is Lipschitz continous.

\item \label{canon-sol-lemma:2} Adding 
%
%
\ref{coeff-assum:4} implies that $g'(x)>0$ for all $x\geq0$.

\item \label{canon-sol-lemma:3} Adding 
\ref{coeff-assum:4}--\ref{coeff-assum:5}  implies that 
$\lim_{x\rightarrow\infty}g'(x)=\infty$.
%
%
%
%
%

\item \label{canon-sol-lemma:4} Adding 
\ref{coeff-assum:4}--\ref{coeff-assum:6} implies that a unique $x_b\in(0,\infty)$ such that $g''(x_b)=0$, $g''(x)<0$ for $x<x_b$ and $g''(x)>0$ for $x>x_b$ exists.

\end{enumerate}
\end{lem}

It will in the sections below be shown that both the precommitment and the equilibrium solutions (to be defined) are of the following kind:  
\begin{defi} A  dividend policy of impulse control type $S=(\tau_n,\zeta_n)_{n\geq 1}$, see \eqref{condaaa}, is said to be a constant lump sum dividend barrier policy if:
\begin{itemize} 
\item Each dividend is of the same size, i.e. 
\begin{align} 
\zeta_n=\bar{x}-\underline{x}, \enskip \mbox{for some $\bar{x}>\underline{x}\geq 0$},
\end{align}  
except possibly at time $0$ when a dividend of size $x-\underline{x}$ is paid if $x\geq \bar{x}$.
\item A dividend is paid when the process $X$ reaches a fixed level, i.e. 
\begin{align} 
\tau_1=\inf\{t > 0: X_t\geq \bar{x}\}, \enskip \tau_n=\inf\{t> \tau_{n-1}: X_t\geq \bar{x}\}, \enskip n=2,3...
\end{align}
(we use the convention $\inf \emptyset= \infty$).
\end{itemize}
\end{defi}
In the sequel we will in the case of a dividend policy of impulse control type $S=(\tau_n,\zeta_n)_{n\geq 1}$ write the corresponding value function, cf. \eqref{classical-problem} and \eqref{condaaa}, as 
\begin{align} 
J(x;S)= \mathbb{E}_x\left( \sum_{n:\tau_n\leq \tau}e^{-r\tau_n}\zeta_n\right).\label{value-func}
\end{align}
Moreover, with a slight abuse of notation we denote a constant lump sum dividend barrier policy $S$ by $(\underline{x},\bar{x})$ and write the corresponding value function \eqref{value-func} as $J(x;\underline{x},\bar{x})$, and similarly for e.g. the function $R(x;S)$ defined in \eqref{constraint2}. 
%
%
%
We will use the following results.

\begin{prop} \label{expl-JR} 
Consider an arbitrary constant lump sum dividend barrier policy $(\underline{x},\bar{x})$. The corresponding value function $J(x;S)$ 
 is then continous and can be written as
\begin{equation}
J(x;\underline{x},\bar{x})=\begin{cases}
J^0(x;\underline{x},\bar{x}):=g(x)\frac{\bar{x}-\underline{x}}{g(\bar{x})-g(\underline{x})}, 					&\ 0 \leq x \leq \bar{x},\\
 x-\underline{x}+ J^0(\underline{x};\underline{x},\bar{x}) ,		& x>\bar{x}. \label{expl-JR:J} 
	\end{cases}
	\end{equation}
Moreover, the corresponding function $R(x;S)$ defined in \eqref{constraint2} is continous and  can be written as
\begin{equation}
R(x;\underline{x},\bar{x})=\begin{cases}
R^0(x;\underline{x},\bar{x}):=g(x)\frac{1}{g(\bar{x})-g(\underline{x})}, 					&\ 0 \leq x \leq \bar{x},\\
1 + R^0(\underline{x};\underline{x},\bar{x}) ,		& x>\bar{x}. \label{expl-JR:R}
	\end{cases}
	\end{equation}
	\end{prop}

\begin{lem} \label{propR} \quad 
\begin{enumerate} [label=(\roman*)]
\item \label{propR:part1}
Consider an arbitrary initial surplus $x>0$. 
Then, 
$R(x;\underline{x},\bar{x})$ is continous and strictly decreasing in $\bar{x}$ with $\lim_{\bar{x}\rightarrow \infty}R(x;\underline{x},\bar{x}) = 0$ for any fixed $\underline{x}\geq 0$. 
Moreover,  
$R(x;\underline{x},\bar{x})$ is continous and strictly increasing in $\underline{x}$ with $\lim_{\underline{x}\rightarrow \bar{x}}R(x;\underline{x},\bar{x}) = \infty$ for any fixed $\bar{x}>0$. 
\item \label{propR:part2} $R(\bar{x};\underline{x},\bar{x})$ is continous and strictly decreasing in $\bar{x}$ with 
$\lim_{\bar{x}\rightarrow \infty}R(\bar{x};\underline{x},\bar{x}) = 1$ for any fixed $\underline{x}>0$. $R\left(\bar{x};0,\bar{x}\right)=1$ for any fixed $\bar{x}>0$.
\end{enumerate}

\end{lem}

In Sections \ref{sec:pre-commitment-solution} and \ref{sec:time-consistent-solution} it will be shown that if we let the moment constraint \eqref{constraint2} vanish in the sense of sending $k\rightarrow 0$ then the precommitment and equilibrium solutions both converge to the solution of the classical dividend problem \eqref{classical-problem}; see Corollary \ref{sol-conv-pre} and Theorem \ref{sol-conv-equi}, respectively. 
For the convenience of the reader we therefore include the following result which follows directly from \cite[Theorem 4.3]{shreve1984optimal} and Lemma \ref{canon-sol-lemma}.  
\begin{prop} \label{classical-sol}   
%
The optimal dividend policy for the unconstrained problem \eqref{classical-problem} reflects the state process \eqref{state-process} at the barrier 
\begin{align}
x^*:=x_b \label{star-def}
\end{align}
(where $x_b$ is defined in Lemma \ref{canon-sol-lemma}) while being flat off $\{t\geq0:X_t=x^*\}$; where it shall be understood that if the initial surplus satisfies $x>x^*$ then there is an immediate dividend payment of size $x-x^*$. The optimal value function is 
\begin{equation}
U(x)=\begin{cases} \label{uncon-problembU}
\frac{g(x)}{g'(x^*)}, 					&\ 0 \leq x \leq x^*,\\
x-x^*+ U\left(x^*\right) ,		& x>x^*.
	\end{cases}
	\end{equation} 
\end{prop} 
\begin{rem} \cite[Theorem 4.3]{shreve1984optimal} presents the solution to problem \eqref{classical-problem} under, essentially, 
\ref{coeff-assum:2}--\ref{coeff-assum:3} and the following relaxed version of \ref{coeff-assum:4}:
\begin{align}
\mu'(x)\leq r \enskip \mbox{for all $x\geq0$}\tag{A.3'}.\label{rem-about-fixedcostres0:1}
\end{align}
In particular it, essentially, says that:  
if $x_b=0$ (which is equivalent to $\mu(0)\leq0$, see \cite[Lemma 2.2]{paulsen2007optimal}) then the optimal policy is to pay all initial surplus $x$ as a dividend immediately for all $x$; 
if $x_b\in(0,\infty)$, then the solution is as in  Proposition \ref{classical-sol}; 
if $x_b=\infty$ then no optimal policy exists, but the optimal value function can be obtained by considering the value function for a reflection dividend policy at a barrier $b$ and then sending $b\rightarrow \infty$, i.e.
\begin{align}
U(x)=\frac{g(x)}{\lim_{b\rightarrow\infty} g'(b)}.\label{uncon-problembU-infinite}
\end{align}
\end{rem} 
Let us explain some of the notation used in the present paper: As used above the derivative of a one-dimensional function $f(\cdot)$ is denoted by $f'(\cdot)$. This notation is also used for the derivative of a multi-dimensional function with respect to the first variable in case it is separated from the other variables with a semi-colon; otherwise the derivative with respect to, say $x$, is indicated by a subindex $x$. By way of example, $J'(x;\underline{x},\bar{x})$ is the derivative of $J(x;\underline{x},\bar{x})$ with respect to $x$ while $A_x(x,y)$ is the derivative of $A(x,y)$ with respect to $x$. We also use the general notation $f(x+):=\lim_{y\searrow x}f(y)$ and $f(x-):=\lim_{y\nearrow x}f(y)$.

\section{Precommitment solution} \label{sec:pre-commitment-solution}
We define the precommitment interpretation of the constrained optimal dividend problem as 
\begin{align}   
V(x_0) = \sup_{S\in \mathcal A(x_0,k)}J(x_0,S), 
\enskip \mbox{where $x_0>0$ is arbitrary but fixed.}\label{pre-commitment-prob}
\end{align}
The solution to problem \eqref{pre-commitment-prob} is presented in Theorem \ref{sec:pre-commitment-solution-THM}. Our approach to this constrained problem relies on the Lagrangian idea. If we add the constraint as a penalty term with Lagrange-parameter $\lambda$, the unconstrained optimization problem reads as
\begin{align*}
& J(x_0,S)-\lambda\left(R(x_0,S)-\frac{1}{k}\right)\\
&=\mathbb{E}_{x_0}\left(\sum_{n:\tau_n\leq \tau}e^{-r\tau_n}\zeta_n\right)-\lambda\left(\mathbb{E}_{x_0}\left(\sum_{n:\tau_n\leq \tau}e^{-r\tau_n}\right)-\frac{1}{k}\right)\\
&=\mathbb{E}_{x_0}\left(\sum_{n:\tau_n\leq \tau}e^{-r\tau_n}(\zeta_n-\lambda)\right)-\frac{\lambda}{k}.
\end{align*}
We thus see a natural connection to the well-studied dividend problem in the case a fixed cost $c>0$ is incurred each time a dividend is paid. Therefore, it is not surprising that 
our treatment relies on properties of the fixed cost dividend problem; which in the present setting corresponds to 
\begin{align}
\begin{split}
 \sup_{S\in \mathcal A(x)}H(x;c,S), & \enskip \mbox{where $c>0$ and} \\
 H(x;c,S):&= \mathbb{E}_x\left(\sum_{n:\tau_n\leq \tau}e^{-r\tau_n}(\zeta_n-c)\right)\\ 
    & \left(= J(x;S)-cR(x;S)\right).\label{fixed-cost-prob}
\end{split}
\end{align}
In the case of a constant lump sum dividend barrier policy $(\underline{x},\bar{x})$ we write the function $H(x;c,S)$ as $H(x;c,\underline{x},\bar{x})$ and note that 
\begin{equation}
H(x;c,\underline{x},\bar{x}) = \begin{cases}
g(x) \frac{\bar{x}-\underline{x}-c}{g(\bar{x})-g(\underline{x})}, 					&\ 0 \leq x \leq \bar{x},\\
 x-\underline{x}-c+ g(\underline{x}) \frac{\bar{x}-\underline{x}-c}{g(\bar{x})-g(\underline{x})} ,		& x>\bar{x}. \label{H-function}
\end{cases}
\end{equation}

The solution to problem \eqref{fixed-cost-prob} is presented in Proposition \ref{fix-cost-rem} below; which follows directly from \cite[Theorem 2.1 and Remark 2.2(e)]{paulsen2007optimal} together with Lemma \ref{canon-sol-lemma} (another reference is \cite[Theorem 2.3 and Remark 2.4]{bai2012optimal}); see also Remark \ref{sol-general-fixed-cost-problem}. Recall that \ref{coeff-assum:2}--\ref{coeff-assum:6} are assumed throughout the paper.
%

\begin{prop}\label{fix-cost-rem}  For any fixed $c>0$, a constant lump sum dividend barrier policy independent of $x$ is optimal in \eqref{fixed-cost-prob}. In particular:
\begin{enumerate} [label=(\roman*)]
\item \label{fix-cost-rem:1}
If the (smooth fit) equation system
\begin{align}
H'(\bar{x}-;c,\underline{x},\bar{x})=1, \enskip H'(\underline{x};c,\underline{x},\bar{x})=1, 
\enskip \underline{x}>0,\label{fixed-cost-smoothfit1}
\end{align}
has a solution, denoted by $(\underline{x}_c,\bar{x}_c)$, then it is unique and it is also an optimal constant lump sum dividend barrier policy, i.e. 
\begin{align}
\sup_{S\in \mathcal A(x)}H(x;c,S) = H(x;c,\underline{x}_c,\bar{x}_c), \enskip \mbox{for all $x\geq0$.}
\end{align} 

\item \label{fix-cost-rem:2}
If \eqref{fixed-cost-smoothfit1} does not have a solution then the (smooth fit) equation
\begin{align}
H'(\bar{x}-;c,0,\bar{x})=1, \label{fixed-cost-smoothfit2}
\end{align}
has a unique solution, denoted by $\bar{x}_c$, and the constant lump sum dividend barrier policy 
$(\underline{x}_c,\bar{x}_c)=(0,\bar{x}_c)$ is optimal, i.e.
\begin{align}
\sup_{S\in \mathcal A(x)}H(x;c,S) = H(x;c,0,\bar{x}_c), \enskip \mbox{for all $x\geq0$}.
\end{align}
%
%
%
\end{enumerate} 
  
\end{prop}

Proposition \ref{fix-cost-rem:properties} below presents properties of the solution to problem \eqref{fixed-cost-prob} which we rely on when proving the main results of the present section Theorem \ref{sec:pre-commitment-solution-THM} and Corollary \ref{sol-conv-pre}. Most of these properties have been established before, if not exactly for the problem \eqref{fixed-cost-prob} and setting of the present paper then for similar problems, 
see Remark \ref{rem-about-fixedcostres}.

\begin{prop}\label{fix-cost-rem:properties} The optimal policy in Proposition \ref{fix-cost-rem} has the following properties:
\begin{align}
& \mbox{$0 \leq  \underline{x}_{c}< x^*<\bar{x}_{c}$} \label{sec:pre-commitment-solution:property0}\\
& \mbox{$\bar{x}_{c}$ is continous and increasing in $c$; $\underline{x}_{c}$ is continous and decreasing in $c$} \label{sec:pre-commitment-solution:property1}\\
& \mbox{$\underline{x}_c,\bar{x}_c\rightarrow x^*>0$ as  $c\searrow 0$} \label{sec:pre-commitment-solution:property2}\\
& \mbox{$H(x;c,\underline{x}_c,\bar{x}_c) \rightarrow U(x)$ as  $c\searrow 0$, for any $x \geq 0$}\label{sec:pre-commitment-solution:property4}\\
& \mbox{$\bar{x}_c>c$ and (hence) $\bar{x}_c\rightarrow \infty$ as  $c\rightarrow \infty$} \label{sec:pre-commitment-solution:property3}\\
& \mbox{there exists a $\bar{c}>0$ such that if $c\geq \bar{c}$ then $(\underline{x}_c,\bar{x}_c)$ is a ruin policy, i.e. $\underline{x}_c=0$} \label{sec:pre-commitment-solution:property3.5}
\end{align}
(recall that $x^*$ and $U(x)$ correspond to the solution to problem \eqref{classical-problem}, cf. Proposition \ref{classical-sol}).
 
\end{prop}

We also need the following result.

\begin{lem} \label{lem2} For any fixed initial surplus $x_0>0$ there exists a unique constant $c(x_0,k)>0$ (depending on  $x_0$ and $k$) such that 
$R(x_0;\underline{x}_c,\bar{x}_c) \geq \frac{1}{k}$ for $c\leq c(x_0,k)$ and 
$R(x_0;\underline{x}_c,\bar{x}_c) \leq \frac{1}{k}$ for $c\geq c(x_0,k)$; where we recall that $(\underline{x}_c,\bar{x}_c)$ is determined in  Proposition \ref{fix-cost-rem}. In particular, 
\begin{align}   
R\left(x_0;\underline{x}_{c(x_0,k)},\bar{x}_{c(x_0,k)}\right)= \frac{1}{k}. \label{condition-for-precom-sol}
\end{align}
\end{lem}

Let the constant lump sum dividend barrier policy determined by Proposition \ref{fix-cost-rem} using the cost $c(x_0,k)$ be denoted by 
$(\ut{x}_k, \tilde{x}_k)$, i.e. let 
\begin{align}
(\ut{x}_k, \tilde{x}_k):=\left(\underline{x}_{c(x_0,k)},\bar{x}_{c(x_0,k)}\right). \label{sec:pre-commitment-solution:eq}
\end{align}

\begin{thm} [Precommitment solution] \label{sec:pre-commitment-solution-THM} For any fixed initial surplus $x_0>0$ the constant lump sum dividend barrier policy $(\ut{x}_k, \tilde{x}_k)$ defined in \eqref{sec:pre-commitment-solution:eq} is optimal in \eqref{pre-commitment-prob} and the corresponding optimal precommitment value is 
\begin{equation}
V(x_0) = 
\begin{cases}
\frac{\tilde{x}_k-\ut{x}_k}{k}, 					&\ 0 \leq x_0 \leq \tilde{x}_k,\\
 x_0-\ut{x}_k+ \frac{g\left(\ut{x}_k\right)}{g\left(x_0\right)}\frac{\tilde{x}_k-\ut{x}_k}{k},		& x_0>\tilde{x}_k.\label{precom-value}
	\end{cases}
	\end{equation}
	\end{thm}
\begin{proof} (of Theorem \ref{sec:pre-commitment-solution-THM}) 
Using e.g. Proposition \ref{expl-JR} it is easy to verify that the right side of \eqref{precom-value} is equal to $J\left(x_0;\ut{x}_k,\tilde{x}_k\right)$, i.e. the value obtained when using the constant lump sum dividend barrier policy $(\ut{x}_k, \tilde{x}_k)$, and the second statement therefore follows from the first statement.

Note that $(\ut{x}_k, \tilde{x}_k) \in \mathcal A (x_0;k)$ by \eqref{condition-for-precom-sol}--\eqref{sec:pre-commitment-solution:eq}. Consider an arbitrary fixed dividend policy $\bar S=(\tau_n,\zeta_n)_{n\geq 1} \in \mathcal A(x_0,k)$. 
Using $c(x_0,k)>0$ (Lemma \ref{lem2}) and the constraint \eqref{constraint2} we obtain
\begin{align}   
J(x_0;\bar S) & \leq J(x_0;\bar S) -c(x_0,k)\left(R(x_0;\bar S)-\frac{1}{k}\right)\\
& = \mathbb{E}_{x_0}\left( \sum_{n:\tau_n\leq \tau}e^{-r\tau_n}\left(\zeta_n-c(x_0,k)\right)\right) +c(x_0,k)\frac{1}{k}.
\end{align}
Now use Proposition \ref{fix-cost-rem} and \eqref{H-function} to see that 
\begin{align}   
\sup_{S \in \mathcal A(x_0)}&\mathbb{E}_{x_0}\left( \sum_{n:\tau_n\leq \tau}e^{-r\tau_n}\left(\zeta_n-c(x_0,k)\right)\right)\\
& = H\left(x_0;{c(x_0,k)},\underline{x}_{c(x_0,k)},\bar{x}_{c(x_0,k)}\right)\\
& = J\left(x_0;\underline{x}_{c(x_0,k)},\bar{x}_{c(x_0,k)}\right)- c(x_0,k)R\left(x_0;\underline{x}_{c(x_0,k)},\bar{x}_{c(x_0,k)}\right).
\end{align}
Hence, using Lemma \ref{lem2} again we obtain
\begin{align}   
J(x_0;\bar S) & \leq 
J\left(x_0;\underline{x}_{c(x_0,k)},\bar{x}_{c(x_0,k)}\right)- c(x_0,k)R\left(x_0;\underline{x}_{c(x_0,k)},\bar{x}_{c(x_0,k)}\right)+c(x_0,k)\frac{1}{k}\\
& = J\left(x_0;\underline{x}_{c(x_0,k)},\bar{x}_{c(x_0,k)}\right).
\end{align}
\end{proof}

\begin{rem} \label{pre-commitment-prob-is-time-incons} The barrier $\tilde{x}_k$ and the dividend $\tilde{x}_k-\ut{x}_k$ in the optimal policy for problem \eqref{pre-commitment-prob}, see Theorem \ref{sec:pre-commitment-solution-THM}, depend on the initial surplus $x_0$ and it is therefore clear that the optimal dividend policy $(\ut{x}_k,\tilde{x}_k)$ chosen at time $0$ will not generally be optimal at a time $t>0$ when the constraint \eqref{constraint2} is updated with the value for the state process observed at $t$, and that the problem of the present paper is in this sense time-inconsistent. 
We remark that Figure \ref{fig-precom2} in Section \ref{example:sec} illustrates how the precommitment value $V(x_0)$ depends on $x_0$ in a specific example. 
We remark that this inconsistency has to do with the discounted reward criterion considered here. If we instead consider the (economically perhaps not too meaningful) long term average criterion
	\[\liminf_{T\to\infty}\frac{1}{T}\mathbb{E}_x\left(D_T\right)\]
	with corresponding constraint 
	\[\limsup_{T\to\infty}\frac{1}{T}\mathbb{E}_{x}\left( \sum_{n:\tau_n\leq T}1\right)\left[=\limsup_{T\to\infty}\frac{1}{T}\mathbb{E}_{x} |\{n:\tau_n\leq T\}|\right] \leq \frac{1}{k},\]
	then the solution to the problem will become independent of the initial state $x$ due to the ergodicity of the situation.  We refer to \cite{christensen2012optimal} for the treatment of a problem of this type in portfolio optimization. 
\end{rem}

\begin{rem} \label{how-to-find-precom-sol} Consider an arbitrary fixed initial surplus $x_0>0$. Relying on e.g. 
Lemma \ref{propR}\ref{propR:part1}, 
Lemma \ref{lem2} and 
Theorem \ref{sec:pre-commitment-solution-THM} 
it is easy to see that the optimal precommitment policy $(\ut{x}_k,\tilde{x}_k)$ can be determined as follows:  
\begin{itemize} 
\item Pick a cost $c_1>0$. 
\item Determine (numerically) $\left(\underline{x}_{c_1},\bar{x}_{c_1}\right)$ according to Proposition \ref{fix-cost-rem}. 
\item 
If $R\left(x_0;\underline{x}_{c_1},\bar{x}_{c_1}\right) > \frac{1}{k}$, we know that $c_1$ is too low, i.e. $c_1<c(x_0,k)$, and we set $c_1=c_1^l$ and choose a new $c_2>c^l_1$. 
If $R\left(x_0;\underline{x}_{c_1},\bar{x}_{c_1}\right) < \frac{1}{k}$ then we know that $c_1$ is too high, i.e.  $c_1>c(x_0,k)$, and we set $c_1=c_1^h$ and choose a new $c_2<c^h_1$.

\item Iterate the steps above always choosing (using e.g. the bisection method) $c_{n+1}$ larger than all $c_i^l,i\leq n$ and smaller than all $c_i^h,i\leq n$ until you find a $c$ which attains \eqref{condition-for-precom-sol} i.e. with $R\left(x_0;\underline{x}_c,\bar{x}_c\right) = \frac{1}{k}$; now set 
$(\ut{x}_k,\tilde{x}_k)=\left(\underline{x}_{c},\bar{x}_{c}\right)$, then this is the optimal precommitment policy given the initial surplus $x_0$.
\end{itemize}
We remark that more numerically efficient methods to find the optimal precommitment policy are of course likely to exist. 
\end{rem} 

The following result follows directly from 
Lemma \ref{propR}\ref{propR:part1}, Proposition \ref{fix-cost-rem:properties} and Lemma \ref{lem2}:

\begin{cor} \label{sol-conv-pre} For any fixed initial surplus $x_0>0$ the precommitment solution $(\ut{x}_k,\tilde{x}_k)$ in Theorem \ref{sec:pre-commitment-solution-THM} 
has the following properties: 
\begin{align}
& \mbox{$0 \leq  \ut{x}_k< x^*<\tilde{x}_k$}\label{sol-conv-pre:1} \\
& \mbox{$\tilde{x}_k$ is continous and increasing in $k$; $\ut{x}_k$ is continous and decreasing in $k$}\label{sol-conv-pre:2}\\
& \mbox{$\ut{x}_k,\tilde{x}_k \rightarrow x^*$ as $k \rightarrow 0$}\label{sol-conv-pre:3}\\
& \mbox{$V(x_0) \rightarrow U(x_0)$ as $k \rightarrow 0$}\label{sol-conv-pre:4}\\
& \mbox{$\tilde{x}_k \rightarrow \infty$ as $k \rightarrow \infty$}\label{sol-conv-pre:5}\\
& \mbox{there exists a $\bar{k}>0$ such that if $k\geq \bar{k}$ then $(\ut{x}_k,\tilde{x}_k)$ is a ruin policy, i.e. $\ut{x}_k =0$.}
\end{align}
\end{cor}

\begin{rem} 
The main interpretations of Corollary \ref{sol-conv-pre} are:
(1) relaxing the constraint \eqref{constraint2} in the sense of sending $k\rightarrow0$ implies that the optimal precommitment solution for the constrained problem \eqref{pre-commitment-prob} converges monotonically to the solution of the unconstrained problem \eqref{classical-problem}, cf. Proposition \ref{classical-sol}, and 
(2) if the constraint is sufficiently restrictive, i.e. if  $k$ is sufficiently large, then a ruin policy is optimal. 
%
%
%
%
\end{rem}

\begin{rem}\label{sol-general-fixed-cost-problem} In \cite{paulsen2007optimal} problem \eqref{fixed-cost-prob} is solved under essentially \ref{coeff-assum:2}--\ref{coeff-assum:3} %
%
%
and \eqref{rem-about-fixedcostres0:1}. Under these less restrictive assumptions it may be that neither case \ref{fix-cost-rem:1} nor  case \ref{fix-cost-rem:2} in Proposition \ref{fix-cost-rem} hold; and in this case, according to \cite[Theorem 2.1]{paulsen2007optimal}, it holds that no optimal policy exists, but the optimal value function can be obtained by considering the value function for a reflection dividend policy at a barrier $b$ and then sending $b\rightarrow \infty$, as in \eqref{uncon-problembU-infinite}.
\end{rem}

\begin{rem} \label{rem-about-fixedcostres}
Properties \eqref{sec:pre-commitment-solution:property0} and \eqref{sec:pre-commitment-solution:property1} are in a setting similar to that of the present paper established in \cite[p. 675]{paulsen2007optimal}. 
Property \eqref{sec:pre-commitment-solution:property2} is established in \cite[Remark 2]{jeanblanc1995optimization} for a model based on a Wiener process with drift. 
\end{rem}

\section{Time-consistent solution} \label{sec:time-consistent-solution}
Let us start by defining the notions of a pure dividend strategy and a (pure subgame perfect Nash) equilibrium. 
These definitions are motivated in Section \ref{discussion-eq}, which also contains a discussion of the results in this section. 

\begin{defi} \label{pure-mix-def} An impulse control policy $S=(\tau_n,\zeta_n)_{n\geq 1}$, see \eqref{condaaa}, is said to be a pure Markov dividend strategy profile if it for each $x\geq 0$ holds that:
\begin{itemize}
\item Each dividend date is an exit time from a set $\mathcal{W} \subseteq [0,\infty)$ that is open in $[0,\infty)$, i.e. 
\begin{align} 
\tau_1=\inf\{t \geq 0: X_t \notin \mathcal{W} \}, \enskip \tau_n=\inf\{t> \tau_{n-1}: X_t \notin \mathcal{W}\}, \enskip n=2,3...
\end{align}
\item Each dividend is given by $\zeta_n =\zeta(X_{\tau_n})$, for some measurable function 
$\zeta(\cdot)$ satisfying $x-\zeta(x)\in \{0\} \cup  \mathcal{W} $ for all $x\notin \mathcal{W}$.
%
%
\end{itemize}
\end{defi}
From now on we refer to a pure Markov dividend strategy profile as a pure dividend strategy. 
%
%
%
%
%
%

\begin{defi} [Equilibrium] \label{equilibrium-def} A pure dividend strategy $\hat S$ is said to be a (pure subgame perfect Nash) equilibrium if for all $x>0$: 
\begin{align}
& \mbox{$\hat S \in \mathcal A(x,k)$},\tag{EqI}\label{eq:i}\\ 
\tag{EqII}\label{eq:ii} 
\begin{split}
& \mbox{for the process $X$ satisfying \eqref{state-process} with $D_t=0$ for all $t\geq0$ it holds that} \\
& \quad \quad  \quad   \quad \quad \quad \liminf_{h \searrow 0 }\frac{J\left(x;\hat S \right)-\mathbb{E}_x\left(e^{-r\tau_h}J\left(X_{\tau_h};\hat S\right)\right)}{\mathbb{E}_x\left(\tau_h\right)}\geq 0, 
\end{split}\\
& \mbox{for all $y\in[0,x]$: if $R\left(y;\hat S\right)\leq \frac{1}{k}-1$ then }J\left(x;\hat S\right) \geq J\left(y;\hat S\right) + x-y. \tag{EqIII}\label{eq:iii}
\end{align}
Moreover, $J\left(\cdot;\hat S \right)$ is said to be the equilibrium value function (corresponding to $\hat S$).
\end{defi}
%
%

\begin{rem} \label{equi-k-bigger-than-1} If $k>1$ then the only strategy that satisfies \eqref{eq:i} for each $x$ is to never pay dividends
; and it is easy to see that this is the unique equilibrium in this case.
\end{rem}
With the remark above in mind we assume in the rest of this section that
\begin{align}
k\leq 1.\tag{A.6}
\end{align}  
Recall also that \ref{coeff-assum:2}--\ref{coeff-assum:6} are assumed throughout the paper. We need the following result. 

\begin{lem} \label{lemma-equi} \quad
\begin{enumerate} [label=(\roman*)]
\item \label{lemma-equi:part1}
The equation system
\begin{align}
R\left({\bar{x}};{\underline{x}},{\bar{x}}\right) &= \frac{1}{k}\label{eq-eq2:constraint} \\
J'\left({\bar{x}}-;{\underline{x}},{\bar{x}}\right) &= 1\label{eq-eq2:smooth-fit}
\end{align} 
has a unique solution $(\underline{x},\bar{x})$, for which it holds that $0\leq \underline{x}<x^*<\bar{x}$. Moreover, if $k=1$ then $\underline{x}=0$ and if $k<1$ then $\underline{x}>0$.
\item \label{lemma-equi:part2}
Let $(\underline{x},\bar{x})$ be the unique solution to \eqref{eq-eq2:constraint}--\eqref{eq-eq2:smooth-fit}. Then, $\bar{x}$ is determined by the (smooth fit) equation 
\begin{align}
J'\left({\bar{x}}-;\bar{x} -k\frac{g(\bar{x})}{g'(\bar{x})},{\bar{x}}\right) = 1,\label{lemma-equi2:1} 
\end{align} 
and
\begin{align}
\underline{x} = \bar{x}-k\frac{g(\bar{x})}{g'(\bar{x})}\label{lemma-equi2:eq1}.
\end{align}

\item  \label{lemma-equi:part3} Equation \eqref{lemma-equi2:1} simplifies to
\begin{align}
g\left(\bar{x} -k\frac{g(\bar{x})}{g'(\bar{x})}\right) = (1-k)g(\bar{x}). \label{lemma-equi2:1-ver2}
\end{align}
\end{enumerate}

\end{lem}

\begin{thm} [Time-consistent solution] \label{sol-equi} \quad
\begin{enumerate}[label=(\roman*)]

\item \label{sol-equi:it1} The constant lump sum dividend barrier strategy $(\utt x_k,\hat x_k)$, where $\hat x_k$ is determined by the (smooth fit) equation \eqref{lemma-equi2:1} and 
$\utt x_k = \hat x_k-k\frac{g(\hat x_k)}{g'(\hat x_k)}$, is an equilibrium.  

\item\label{sol-equi:it3} The equilibrium value function is given by

\begin{equation}
J(x;\utt x_k,\hat x_k) =
\begin{cases}
\frac{g(x)}{g'(\hat x_k)}, 					&\ 0 \leq x \leq \hat x_k,\\
 x-\utt x_k+ \frac{g(\utt x_k)}{g'(\hat x_k)},		& x>\hat x_k. \label{equilib-value-func}
	\end{cases}
	\end{equation}
	\end{enumerate}
\end{thm} 
 
\begin{proof} (of Theorem \ref{sol-equi}) Item \ref{sol-equi:it3} follows from \ref{sol-equi:it1}, Proposition \ref{expl-JR}  and Lemma \ref{lemma-equi}. Let us prove \ref{sol-equi:it1}: 
Condition \eqref{eq:i} (in Definition \ref{equilibrium-def}) is directly verified for each $x$ using that $(\utt x_k,\hat x_k)$ satisfies \eqref{eq-eq2:constraint} and $\utt x_k\geq0$ (Lemma \ref{lemma-equi}). 
 By Lemma \ref{lemma-equi} holds 
\begin{align}
J'\left(\hat x_k-;\utt x_k,\hat x_k\right)
=(J^0)'\left(\hat x_k;\utt x_k,\hat x_k\right)
=J'\left(\hat x_k+;\utt x_k,\hat x_k\right)
=1.\label{mainthmeq1}
\end{align} 
Using e.g. \eqref{expl-JR:J}, \eqref{star-def} and Lemma \ref{canon-sol-lemma}  it is easy to verify that 
\begin{align}
\left(A_X-r\right)J^0\left(\hat x_k;\utt x_k,\hat x_k\right)=0, \mbox{ and } 
(J^0)''\left(x;\utt x_k,\hat x_k\right)\geq 0 \mbox{ for } x\geq x^*.\label{mainthmeq2}
\end{align}
Recall that $\hat x_k> x^*$ (Lemma \ref{lemma-equi}) and that $r>0$. Using the above and also $\mu'(x)-r\leq 0$ (Assumption \ref{coeff-assum}) we find that 
for any $x >\hat x_k$ it holds that 
\begin{align}
&\mu(x)- r\left( x- \hat x_k + J\left(\hat x_k;\utt x_k,\hat x_k \right)\right)\\
&\leq \mu\left(\hat x_k\right)-rJ\left(\hat x_k;\utt x_k,\hat x_k \right)\\
&<\mu\left(\hat x_k\right)(J^0)'\left(\hat x_k;\utt x_k,\hat x_k\right)
+ \frac{1}{2}\sigma^2\left(\hat x_k\right)(J^0)''\left(\hat x_k;\utt x_k,\hat x_k\right)
-rJ\left(\hat x_k;\utt x_k,\hat x_k \right)\\
&= \left(A_X-r\right)J^0\left(\hat x_k;\utt x_k,\hat x_k \right)= 0.
\end{align}%
Using the observations above and \eqref{expl-JR:J} we obtain 
\begin{equation}
\left(A_X-r\right)J\left(x;\utt x_k,\hat x_k \right)  =
\begin{cases} 
0, 					&\ 0 \leq x < \hat x_k,\\
\mu(x)-r\left(x-\utt x_k+J^0\left(\utt x_k;\utt x_k,\hat x_k \right)\right),		& x>\hat x_k,
	\end{cases}
	\end{equation}
\begin{equation}
\quad \quad \quad \quad \quad \quad \quad \quad \quad \enskip  = \begin{cases} 
0, 					&\ 0 \leq x < \hat x_k,\\
\mu(x)-r\left(x-\hat x_k+J^0\left(\hat x_k;\utt x_k,\hat x_k \right)\right),		& x>\hat x_k,
	\end{cases}
\end{equation}
\begin{align}
&\leq 0.\quad\quad\quad\quad\quad\quad\quad\quad\quad\quad\enskip\enskip\enskip\enskip\enskip
\end{align}
Using also a generalized It\^{o} formula, see e.g. \cite[Section 3.5]{peskir2006optimal}, we find that the numerator in \eqref{eq:ii} satisfies 
%
%
\begin{align}
& J\left(x;\utt x_k,\hat x_k \right)- \mathbb{E}_x\left(e^{-r\tau_h}J\left(X_{\tau_h};\utt x_k,\hat x_k\right)\right)\\
& = -\mathbb{E}_x\left(\int_0^{\tau_h}I_{\{X_s\neq \hat x_k\}}e^{-rs}\left(A_X-r\right)J\left(X_s;\utt x_k,\hat x_k\right)ds\right)\\
& \geq 0, \enskip \mbox{for any $h>0$ and any $x>0$.}
\end{align}
This implies that \eqref{eq:ii} holds for each $x$. 
Use that 
$R\left(\hat x_k;\utt x_k,\hat x_k\right)=\frac{1}{k}$, 
$R\left(\utt x_k;\utt x_k,\hat x_k\right)+1=\frac{1}{k}$  
and $R\left(y;\utt x_k,\hat x_k\right)$ is increasing in $y$ to see that
\begin{align}
R\left(y;\utt x_k,\hat x_k\right) \leq \frac{1}{k}-1 \Rightarrow y \leq \utt x_k. \label{mainthmeq3}
\end{align}
Let 
$K(x):= J\left(x;\utt x_k,\hat x_k\right)-x$. It is easy to verify that 
$K(0)=0$, 
$K(\utt x_k)=K(\hat x_k)$, 
$K'(\hat x_k)=0$,
$K''(x)<0$ for $x<x^*$ and $K''(x)>0$ for $x>x^*$. Using also that $\utt x_k<x^*<\hat x_k$ it is easy to see that $K(\cdot)$ is increasing on $[0,\utt x_k)$ and that 
%
%
\begin{align}
\label{mainthmeq4}
\begin{split}
&\mbox{if $0\leq y \leq \utt x_k$ and $x \geq y$ then:}\\
& K(x)-K(y) = J\left(x;\utt x_k,\hat x_k\right)-x - \left(J\left(y;\utt x_k,\hat x_k\right)-y\right) \geq 0.
\end{split}
\end{align}
From \eqref{mainthmeq3} and \eqref{mainthmeq4} it follows that \eqref{eq:iii} holds.
\end{proof}
 
\begin{rem} \label{how-to-find-equi-sol} The equilibrium $(\utt x_k,\hat x_k)$ can be found  as follows. 
First, find the unique solution to \eqref{lemma-equi2:1-ver2} and set $\hat x_k$ equal to this solution. This ensures that the smooth fit equilibrium condition \eqref{lemma-equi2:1} is satisfied. Second, set $\utt x_k = \hat x_k - k \frac{g(\hat x_k)}{g'(\hat x_k)}$. 
\end{rem}

The following definition corresponds to an adaptation of the notion of a strong equilibrium, see Section \ref{discussion-eq} for a motivation. Note that a strong equilibrium is necessarily an equilibrium in the sense of Definition \ref{equilibrium-def}. 


 \begin{defi}[Strong equilibrium] \label{strongequilibrium-def} An equilibrium $\hat S$ is strong if condition \eqref{eq:ii} in Definition \ref{pure-mix-def} can be replaced by:
\begin{align}
\tag{EqII'}\label{eq:ii'}
\begin{split}
& \mbox{for the process $X$ satisfying \eqref{state-process} with $D_t=0$ for all $t\geq0$ there exists} \\
& \mbox{an $\bar{h}>0$ s.t. } J\left(x;\hat S \right)\geq\mathbb{E}_x\left(e^{-r\tau_h}J\left(X_{\tau_h};\hat S\right)\right), \enskip \mbox{for all $h\in [0,\bar{h}]$.}
\end{split}
\end{align}
\end{defi}
It is easy to see that the arguments that imply that condition \eqref{eq:ii} holds in the proof of Theorem \ref{sol-equi} also imply that the stronger condition \eqref{eq:ii'} holds and we therefore obtain:

\begin{thm} \label{strong-eq-thm} 
The equilibrium $(\utt x_k,\hat x_k)$ in Theorem \ref{sol-equi} is strong.
\end{thm}

We also find:

\begin{thm} \label{sol-conv-equi} 
The equilibrium $(\utt x_k,\hat x_k)$ in Theorem \ref{sol-equi} has the following properties:
\begin{align}
& \mbox{$0 \leq  \utt x_k< x^*<\hat x_k$}\label{sol-conv-equi:1}\\
& \mbox{$\hat x_k$ is continous and increasing in $k$; $\utt x_k$ is continous and decreasing in $k$}\label{sol-conv-equi:2}\\
& \mbox{$\utt x_k,\hat x_k \rightarrow x^*$ as $k \rightarrow 0$}\label{sol-conv-equi:3}\\
& \mbox{$J\left(x;\utt x_k,\hat x_k\right)  \rightarrow U(x)$ as $k \rightarrow 0$, for any $x\geq0$}\label{sol-conv-equi:4}\\
& \mbox{the equilibrium is a ruin strategy, i.e. $\utt x_k=0$, if and only if $k=1$.}\label{sol-conv-equi:6}
\end{align}
\end{thm}


\subsection{Motivation of equilibrium definition and discussion} \label{discussion-eq}

The maximization of the value function $J(x;S)$ in \eqref{value-func} under the constraint \eqref{constraint2} is time-inconsistent in the sense that the optimal policy depends on $x$, see Remark \ref{pre-commitment-prob-is-time-incons}. 
The game-theoretic approach is to suppose that the decision maker, or controller, in a time-inconsistent problem is a person with time-inconsistent preferences and to reinterpret the problem as an intrapersonal sequential dynamic game. For the problem of the present paper this means identifying each $x\in(0,\infty)$ with an agent, called the $x$-agent, who decides if, and of which size, a dividend should be paid if the current value of the process $X$ is $x$; and letting all $x$-agents play a sequential dynamic game against each other regarding how to pay dividends from the surplus process $X$.  
%
%
(We remark that similar interpretations for regular time-inconsistent stochastic control can be found in e.g. \cite{bjork2017time,tomas-disc,lindensjo2017timeinconHJB}, and in \cite{christensen2017finding,christensen2018time} for time-inconsistent stopping.)  The interpretation of Definition \ref{pure-mix-def} is therefore that each $x$-agent must make his decision at $x$ without randomization and only based the current value $x$ and in this sense it is clear that Definition \ref{pure-mix-def} corresponds to a pure Markov strategy profile. (Recall that, in general, a pure Markov strategy depends only on past events that are payoff relevant and determines the action of an agent without randomization and that a strategy profile is a complete specification of the strategies of every agent in a game.) 

The items in Definitions \ref{equilibrium-def} and \ref{strongequilibrium-def} have the following interpretations:
 
\begin{itemize}
\item   Condition \eqref{eq:i} ensures that an equilibrium $\hat S$ is admissible from the viewpoint of every $x$-agent. In particular, the constraint \eqref{constraint2} is, from the view-point of every $x$-agent, satisfied.

\item   Condition  \eqref{eq:ii} is an adaptation of the usual first order equilibrium condition in time-inconsistent stochastic control, studied for a general model in continuous time in \cite{bjork2017time}. 
The interpretation of \eqref{eq:ii} is that an $x$-agent's criterion for not deviating from $\hat S$ by not paying a dividend at $x$ when $\hat S$ prescribes paying a dividend at $x$ is that the instantaneous expected discounted rate of change relative to $\E_x(\tau_h)$ obtained by deviating is non-positive. In line with \cite[Remark 3.5]{bjork2017time} and \cite[Section 2.1]{christensen2018time} we remark that this kind of first-order equilibrium may correspond to a stationary point that is not a maximum in the sense that the numerator in \eqref{eq:ii} can be negative for each fixed $h>0$ and still be in line with \eqref{eq:ii} by vanishing with order $o(\E_x(\tau_h))$. Clearly, it is not entirely satisfactory in every situation that an equilibrium may correspond to a stationary point that is not a maximum in this sense. With an observation of this kind as a motivation the notion of a strong equilibrium was, in a time-inconsistent stochastic control framework, defined in \cite{huang2018strong}. \eqref{eq:ii'} is an adaptation of this notion to the problem of the present paper and the interpretation is, based on the discussion above, obvious.

\item   Condition \eqref{eq:iii} means that if it is for an $x$-agent admissible to pay a dividend, then the action prescribed for the $x$-agent by $\hat S$ is more desirable, from the viewpoint of that $x$-agent, than paying any (alternative) admissible divided.

\end{itemize}

A conclusion of Theorem \ref{sol-conv-equi} is that the equilibrium solution for the constrained problem converges monotonically to the optimal solution for the unconstrained problem when the constraint \eqref{constraint2} vanishes in the sense of sending $k\rightarrow0$. It is clearly desirable that an equilibrium solution is optimal in the absence of time-inconsistency and the just mentioned fact thus further supports our equilibrium definition.

Equation \eqref{lemma-equi2:1} can be said to be a smooth fit equilibrium condition. We remark that a smooth fit equilibrium condition for a time-inconsistent stopping problem was found in \cite{christensen2018time}.

\section{An example} \label{example:sec}
In this section we suppose that the uncontrolled surplus process is a Wiener process with positive drift, and let $\mu>0$ and $\sigma>0$ denote the its drift and volatility, respectively. Using elementary calculations, see e.g. \cite[Section 5]{shreve1984optimal}, we find the canonical solution of \eqref{ODE1}--\eqref{ODE2} to be 
\begin{align}
g(x)&=e^{\alpha_1 x}-e^{\alpha_2x}\\
&\enskip \mbox{ where
$\alpha_1:= -\frac{\mu}{\sigma^2} + \sqrt{\frac{\mu^2}{\sigma^4}+\frac{2r}{\sigma^2}}$ and 
$\alpha_2:= -\frac{\mu}{\sigma^2} - \sqrt{\frac{\mu^2}{\sigma^4}+\frac{2r}{\sigma^2}}$.}
\end{align}
It is now easy, cf. \eqref{star-def}, to verify that
\begin{align}
x^* = \frac{\log(\alpha_2^2/\alpha_1^2)}{\alpha_1-\alpha_2}\in(0,\infty).\label{uncon-problemb**}
\end{align}
In all illustrations in this section we use the parameter values 
$\mu=0.06$, 
$\sigma^2=0.03$ and 
$r=0.02$. This implies that $x^*=1.1405$. The strict concavity-convexity of $g(\cdot)$, in the sense of Lemma \ref{canon-sol-lemma}\ref{canon-sol-lemma:4}, is illustrated in Figure \ref{fig-g}.  
\pgfmathsetmacro{\ao}{0.309401077} 
\pgfmathsetmacro{\at}{-4.309401077} 

\begin{figure}[H]
\center
\begin{tikzpicture}[scale=0.85,
declare function={  g(\x)= exp(\x*\ao)-exp(\x*\at);}] 
 \begin{axis}[ymin=0, ymax=3, xlabel={$x$}, ylabel={$g(x)$},smooth]
		\addplot[style=solid,domain=0:5]{g(x)};
				\end{axis} 
	\end{tikzpicture} 
					\caption{The canonical solution of \eqref{ODE1}--\eqref{ODE2} corresponding to a Wiener process with positive drift.}\label{fig-g}
\end{figure}
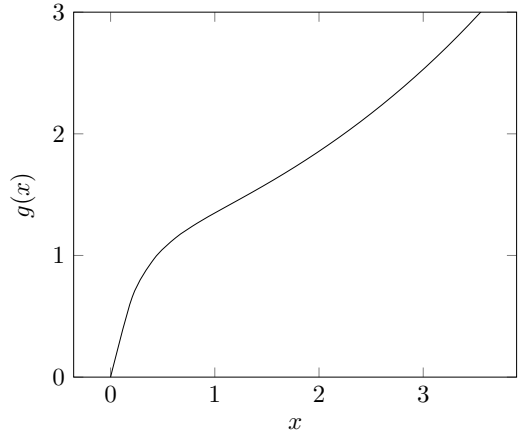

Let us first see how the optimal precommitment policy $(\ut{x}_k,\tilde{x}_k)$ for an arbitrary fixed initial surplus $x_0>0$ is found.  First, use \eqref{H-function} to find
\begin{equation}
H(x;c,\underline{x},\bar{x}) = \begin{cases}
(e^{\alpha_1 x}-e^{\alpha_2x})\frac{\bar{x}-\underline{x}-c}{e^{\alpha_1 \bar{x}}-e^{\alpha_2\bar{x}}-e^{\alpha_1 \underline{x}}+e^{\alpha_2\underline{x}}}, 					&\ 0 \leq x \leq \bar{x},\\
 x-\underline{x}-c+ (e^{\alpha_1 \underline{x}}-e^{\alpha_2\underline{x}})\frac{\bar{x}-\underline{x}-c}{e^{\alpha_1 \bar{x}}-e^{\alpha_2\bar{x}}-e^{\alpha_1 \underline{x}}+e^{\alpha_2\underline{x}}},		& x>\bar{x}. 
\end{cases}
\end{equation}
This implies that \eqref{fixed-cost-smoothfit1} is satisfied if 
\begin{align}
(\alpha_1e^{\alpha_1 \bar{x}}-\alpha_2e^{\alpha_2\bar{x}})\frac{\bar{x}-\underline{x}-c}{e^{\alpha_1 \bar{x}}-e^{\alpha_2\bar{x}}-e^{\alpha_1 \underline{x}}+e^{\alpha_2\underline{x}}} & = 1,\label{exampl-eq:1}\\
(\alpha_1e^{\alpha_1 \underline{x}}-\alpha_2e^{\alpha_2\underline{x}})\frac{\bar{x}-\underline{x}-c}{e^{\alpha_1 \bar{x}}-e^{\alpha_2\bar{x}}-e^{\alpha_1 \underline{x}}+e^{\alpha_2\underline{x}}} & = 1, \enskip \mbox{for $\underline{x}>0$},\label{exampl-eq:2}
\end{align}
and that \eqref{fixed-cost-smoothfit2} is satisfied  if 
\begin{align}
(\alpha_1e^{\alpha_1 \bar{x}}-\alpha_2e^{\alpha_2\bar{x}})\frac{\bar{x} -c}{e^{\alpha_1 \bar{x}}-e^{\alpha_2\bar{x}}} = 1.\label{exampl-eq:3}
\end{align}
Second, we similarly obtain
\begin{equation} 
R(x_0;\underline{x},\bar{x})  = \begin{cases}
\frac{e^{\alpha_1 x_0}-e^{\alpha_2x_0}}{e^{\alpha_1 \bar{x}}-e^{\alpha_2\bar{x}}-e^{\alpha_1\underline{x}}+e^{\alpha_2\underline{x}}}, 					&\ 0 \leq x_0 \leq \bar{x},\\
 1+ \frac{e^{\alpha_1 \underline{x}}-e^{\alpha_2\underline{x}}}{e^{\alpha_1 \bar{x}}-e^{\alpha_2\bar{x}}-e^{\alpha_1\underline{x}}+e^{\alpha_2\underline{x}}},		& x_0>\bar{x}.\label{exampl-eq:4}
\end{cases}
\end{equation}
In order to find the optimal precommitment policy $(\ut{x}_k,\tilde{x}_k)$ we now consider some constant $c_1>0$ and determine $\left(\underline{x}_{c_1},\bar{x}_{c_1}\right)$ as the solution to the dividend problem with fixed cost $c_1$ according to Proposition \ref{fix-cost-rem}. This means that we let$\left(\underline{x}_{c_1},\bar{x}_{c_1}\right)$ be the solution to \eqref{exampl-eq:1}--\eqref{exampl-eq:2} if it exists and if it does not then we set $\left(\underline{x}_{c_1},\bar{x}_{c_1}\right) = \left(0,\bar{x}_{c_1}\right)$ where $\bar{x}_{c_1}$ is the solution to \eqref{exampl-eq:3}. 
An illustration is presented in Figure \ref{fig-precom1}.
\begin{figure}[H]
\center
\begin{tikzpicture}[scale=0.85,
declare function={  g(\x)= exp(\x*\ao)-exp(\x*\at); gp(\x)= \ao*exp(\x*\ao)-\at*exp(\x*\at);}] 
		\begin{axis}[xmin=0, ymax=3.5,xlabel={$x$}, ylabel={$H(x;c,\underline{x}_c,\bar{x}_c)$},smooth]
		\pgfmathsetmacro{\c}{0.1} %
		\pgfmathsetmacro{\xunder}{0.767011167}
		\pgfmathsetmacro{\xbar}{1.852844992}
		\addplot[style=solid,domain=0:\xbar]{g(x)*(\xbar-\xunder-\c)/(g(\xbar)-g(\xunder))};
		\addplot[style=solid,domain=\xbar:{\xbar+1}]{x-\xunder-\c+g(\xunder)*(\xbar-\xunder-\c)/(g(\xbar)-g(\xunder))};
		\addplot[style=dashed,domain={\xunder-0.9}:{\xunder+0.9}]{g(\xunder)*(\xbar-\xunder-\c)/(g(\xbar)-g(\xunder)) + x-\xunder};
		\addplot[style=dashed,domain={\xbar-0.9}:{\xbar+0.9}]{g(\xbar)*(\xbar-\xunder-\c)/(g(\xbar)-g(\xunder)) + x-\xbar};
		\node at (1.2,3.1) [above] {$c=0.1$};
		\pgfmathsetmacro{\cTWO}{0.6} %
		\pgfmathsetmacro{\xunderTWO}{0.545345718313056}
		\pgfmathsetmacro{\xbarTWO}{2.97689011549229}
		\addplot[style=solid,domain=0:\xbarTWO]{g(x)*(\xbarTWO-\xunderTWO-\cTWO)/(g(\xbarTWO)-g(\xunderTWO))};
		\addplot[style=solid,domain=\xbarTWO:{\xbarTWO+1}]{x-\xunderTWO-\cTWO+g(\xunderTWO)*(\xbarTWO-\xunderTWO-\cTWO)/(g(\xbarTWO)-g(\xunderTWO))};
		\addplot[style=dashed,domain={\xunderTWO-0.9}:{\xunderTWO+0.9}]{g(\xunderTWO)*(\xbarTWO-\xunderTWO-\cTWO)/(g(\xbarTWO)-g(\xunderTWO)) + x-\xunderTWO};
		\addplot[style=dashed,domain={\xbarTWO-0.9}:{\xbarTWO+0.9}]{g(\xbarTWO)*(\xbarTWO-\xunderTWO-\cTWO)/(g(\xbarTWO)-g(\xunderTWO)) + x-\xbarTWO};
		\node at (3.5,3.0) [above] {$c=0.6$};
		\pgfmathsetmacro{\cTHREE}{2} %
		\pgfmathsetmacro{\xunderTHREE}{0.31830611359324}
		\pgfmathsetmacro{\xbarTHREE}{4.95798038852098}
		\addplot[style=solid,domain=0:\xbarTHREE]{g(x)*(\xbarTHREE-\xunderTHREE-\cTHREE)/(g(\xbarTHREE)-g(\xunderTHREE))};
		\addplot[style=solid,domain=\xbarTHREE:{\xbarTHREE+1}]{x-\xunderTHREE-\cTHREE+g(\xunderTHREE)*(\xbarTHREE-\xunderTHREE-\cTHREE)/(g(\xbarTHREE)-g(\xunderTHREE))};
		\addplot[style=dashed,domain={\xunderTHREE-0.9}:{\xunderTHREE+0.9}]{g(\xunderTHREE)*(\xbarTHREE-\xunderTHREE-\cTHREE)/(g(\xbarTHREE)-g(\xunderTHREE)) + x-\xunderTHREE};
		\addplot[style=dashed,domain={\xbarTHREE-0.9}:{\xbarTHREE+0.9}]{g(\xbarTHREE)*(\xbarTHREE-\xunderTHREE-\cTHREE)/(g(\xbarTHREE)-g(\xunderTHREE)) + x-\xbarTHREE};
		\node at (5.05,2.6) [above] {$c=2.0$};
		%
	  %
		%
		%
		\end{axis} 
	\end{tikzpicture}
		\caption{The optimal value function of the fixed cost dividend problem, see \eqref{fixed-cost-prob}, for 
		$c=0.1$, $c=0.6$ and $c=2.0$
		for which the optimal dividend policy is 
		$(0.7670,1.8528)$,
		$(0.5453,2.9769)$ and 
		$(0.3183,4.9580)$, respectively. 
		The dashed lines indicate smooth fit. 
		}
		\label{fig-precom1}
\end{figure}
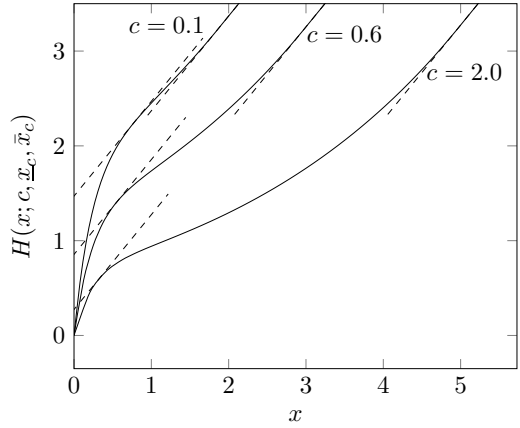  
We now evaluate the function in \eqref{exampl-eq:4} with $(\underline{x},\bar{x})=\left(\underline{x}_{c_1},\bar{x}_{c_1}\right)$  and then iterate this procedure with higher or lower costs in accordance with Remark \ref{how-to-find-precom-sol} until we find a cost $c$ such that \eqref{exampl-eq:4} evaluated at $(\underline{x}_c,\bar{x}_c)$ is equal to $\frac{1}{k}$; which then means that $(\ut{x}_k,\tilde{x}_k)=(\underline{x}_c,\bar{x}_c)$ for the particular initial surplus $x_0$ at hand. An illustration of the precommitment value as a function of $x_0$ is presented in Figure \ref{fig-precom2}. 

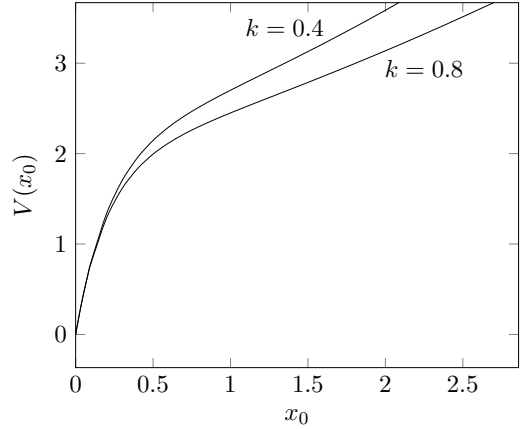
\begin{figure}[H]
\center
\begin{tikzpicture}[scale=0.85] 
		\begin{axis}[xmin=0, ymax=3.67,xlabel={$x_0$}, ylabel={$V(x_0)$},smooth]
				\addplot[style=solid]		
		coordinates{
(0,0)
(0.025,0.232755453458332)
(0.05,0.44298934171257)
(0.075,0.632699127304428)
(0.1,0.803856424058494)
(0.2,1.33855214997862)
(0.3,1.70165920489287)
(0.4,1.95743951096098)
(0.5,2.14611371519001)
(0.6,2.29279989454614)
(0.7,2.41330950829234)
(0.8,2.51768450868222)
(0.9,2.61233938222563)
(1,2.70139459726172)
(1.1,2.78750700819421)
(1.2,2.87239564216954)
(1.3,2.95717690184586)
(1.4,3.04257908551682)
(1.5,3.12907994213907)
(1.6,3.21699541891446)
(1.7,3.30653678680324)
(1.8,3.39784744133373)
(1.9,3.4910267935249)
(2,3.58614547732458)
(2.1,3.68325526287788)
(2.2,3.7823954386352)
(2.3,3.88359690970019)
(2.4,3.98688483407774)
(2.5,4.09228027128356)
(2.6,4.19980139328459)
(2.7,4.30946404102526)
};
\node at (1.35,3.2) [above] {$k=0.4$};
\addplot[style=solid]		
		coordinates{
(0,0)
(0.025,0.232422812367297)
(0.05,0.440757529765052)
(0.075,0.626432053902384)
(0.1,0.791525131935807)
(0.2,1.29011984935864)
(0.3,1.6136791733304)
(0.4,1.834815950251)
(0.5,1.99461327451429)
(0.6,2.11701873646498)
(0.7,2.21644156960143)
(0.8,2.30175391633957)
(0.9,2.37850221042597)
(1,2.45018837358318)
(1.1,2.51903820170936)
(1.2,2.58647104855506)
(1.3,2.6533953165806)
(1.4,2.72039249664035)
(1.5,2.78783662766339)
(1.6,2.85596529820001)
(1.7,2.9249314376208)
(1.8,2.99483231021794)
(1.9,3.06572935437295)
(2,3.13766064235084)
(2.1,3.21064888302772)
(2.2,3.28470953657127)
(2.3,3.35984730130229)
(2.4,3.43606490483103)
(2.5,3.51336203808114)
(2.6,3.59173574432195)
(2.7,3.671181946565)
};
\node at (2.25,2.75) [above] {$k=0.8$};
		\end{axis} 
	\end{tikzpicture}
		\caption{The precommitment value, see \eqref{precom-value}, as a function of $x_0$ for $k=0.4$ and $k=0.8$.} \label{fig-precom2}
\end{figure} 
Let us now find the equilibrium dividend strategy $(\utt x_k,\hat x_k)$. In the rest of the section we suppose that $k\leq 1$, cf. Remark \ref{equi-k-bigger-than-1}. First, note that \eqref{lemma-equi2:1-ver2} becomes 
\begin{align}
e^{\alpha_1 \left(\bar{x}- k \frac{e^{\alpha_1 \bar{x}}-e^{\alpha_2\bar{x}}}{\alpha_1e^{\alpha_1 \bar{x}}-\alpha_2e^{\alpha_2\bar{x}}}\right)}
-e^{\alpha_2 \left(\bar{x}- k \frac{e^{\alpha_1 \bar{x}}-e^{\alpha_2\bar{x}}}{\alpha_1e^{\alpha_1 \bar{x}}-\alpha_2e^{\alpha_2\bar{x}}}\right)}
 = (1-k)\left(e^{\alpha_1 \bar{x}}-e^{\alpha_2\bar{x}}\right). \label{example-eq1}
\end{align}
Following Remark \ref{how-to-find-equi-sol} we now: (1) ensure that the equilibrium smooth fit condition \eqref{lemma-equi2:1} holds by solving \eqref{example-eq1} and setting $\hat x_k$ equal to this solution, and (2) set 
\begin{align}
\utt x_k =  \hat x_k - k \frac{g(\hat x_k)}{g'(\hat x_k)} = 
\hat x_k- k \frac{e^{\alpha_1 \hat x_k}-e^{\alpha_2\hat x_k}}{\alpha_1e^{\alpha_1 \hat x_k}-\alpha_2e^{\alpha_2\hat x_k}}.
\end{align}
An illustration of the equilibrium value function and the equilibrium smooth fit principle is presented in Figure \ref{fig-equi1}.
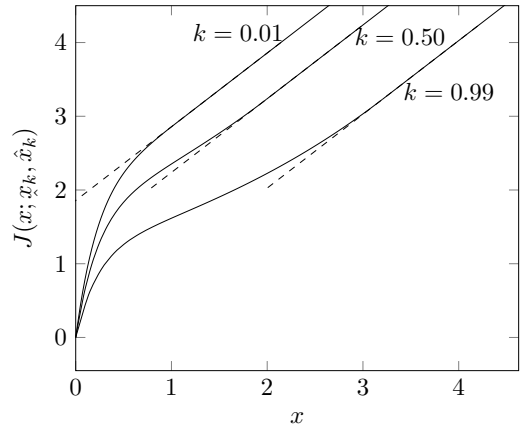
\begin{figure}[H]
\center
	\begin{tikzpicture}[scale=0.85,
declare function={  g(\x)= exp(\x*\ao)-exp(\x*\at); gp(\x)= \ao*exp(\x*\ao)-\at*exp(\x*\at);}] 
	 \begin{axis}[xmin=0, ymax=4.5, xlabel={$x$}, ylabel={$J(x;\utt x_k,\hat x_k)$}, smooth]
			\node at (1.7,3.9) [above] {$k=0.01$};
			\pgfmathsetmacro{\xunder}{1.12055402805}
			\pgfmathsetmacro{\xbar}{1.15065334045}
			\addplot[style=solid,domain=0:\xbar]{g(x)*(\xbar-\xunder)/(g(\xbar)-g(\xunder))};
			\addplot[style=solid,domain=\xbar:100]{x-\xunder+g(\xunder)*(\xbar-\xunder)/(g(\xbar)-g(\xunder))};
			\addplot[style=dashed,domain={\xbar-1.2}:{\xbar+1}]{g(\xbar)*(\xbar-\xunder)/(g(\xbar)-g(\xunder)) + x-\xbar};
			\node at (3.39,3.85) [above] {$k=0.50$};
			\pgfmathsetmacro{\xunderTHREE}{0.37571729589}
			\pgfmathsetmacro{\xbarTHREE}{1.98930000000}
			\addplot[style=solid,domain=0:\xbarTHREE]{g(x)*(\xbarTHREE-\xunderTHREE)/(g(\xbarTHREE)-g(\xunderTHREE))};
			\addplot[style=solid,domain=\xbarTHREE:100]{x-\xunderTHREE+g(\xunderTHREE)*(\xbarTHREE-\xunderTHREE)/(g(\xbarTHREE)-g(\xunderTHREE))};
			\addplot[style=dashed,domain={\xbarTHREE-1.2}:{\xbarTHREE+1}]{g(\xbarTHREE)*(\xbarTHREE-\xunderTHREE)/(g(\xbarTHREE)-g(\xunderTHREE)) + x-\xbarTHREE};
			\node at (3.9,3.1) [above] {$k=0.99$};
			\pgfmathsetmacro{\xunderSIX}{0.00590653977}
			\pgfmathsetmacro{\xbarSIX}{3.20563000000}
			\addplot[style=solid,domain=0:\xbarSIX]{g(x)*(\xbarSIX-\xunderSIX)/(g(\xbarSIX)-g(\xunderSIX))};
			\addplot[style=solid,domain=\xbarSIX:100]{x-\xunderSIX+g(\xunderSIX)*(\xbarSIX-\xunderSIX)/(g(\xbarSIX)-g(\xunderSIX))};
			\addplot[style=dashed,domain={\xbarSIX-1.2}:{\xbarSIX+1}]{g(\xbarSIX)*(\xbarSIX-\xunderSIX)/(g(\xbarSIX)-g(\xunderSIX)) + x-\xbarSIX};
	\end{axis}  
	\end{tikzpicture}
		\caption{The equilibrium value function, see \eqref{equilib-value-func}, for
		$k=0.01$, $k=0.5$ and $k=0.99$, for which the equilibrium dividend strategy is 
		$(1.1206,1.1507)$,
		$(0.3757,1.9893)$ and 
		$(0.0059,3.2056)$, respectively. 
	  The dashed lines indicate smooth fit. }\label{fig-equi1}
\end{figure}

\subsection{Sensitivity with respect to $k$ and $c$}
A natural question when studying impulse control problems is what happens for vanishing fixed costs, in particular what happens to the derivative of the optimal value function. To the best of our knowledge, such an analysis has not been carried out for exactly the optimal dividend problem \eqref{fixed-cost-prob}. 
However, slightly different problems without absorption are studied in \cite{oksendal1999stochastic}, and the references therein, and it is investigated what happens when the fixed cost $c$ is sent to zero, see also \cite{alvarez2008optimal,christensen_irle_ludwig_2017,oksendal2002non}. 
One main finding is that while the value of the problem converges to the one without fixed costs as $c\searrow0$ (as in Proposition \ref{fix-cost-rem:properties}) the derivative with respect to $c$ converges to $-\infty$. The interpretation is that small fixed costs have large effects on the value. 
Figure \ref{fig:asympt_small_costs} illustrates these properties for the optimal dividend problem \eqref{fixed-cost-prob}.  
(In this section we consider a fixed initial surplus $x=x_0=0.025$. The dashed line in each graph below indicates the optimal value without costs, see \eqref{uncon-problembU}, while the dotted line is the optimal dividend barrier without costs, see \eqref{star-def}.)

In contrast to the findings for small fixed costs $c$, Figures \ref{fig:asympt_precom} and \ref{fig:asympt_equili} suggest that small values of $k$ in our constraint \eqref{constraint2} just have small effects on the value for both the precommitment and the equilibrium formulation. To not overburden the present paper, we leave the theoretical investigation of these findings for future research.

\begin{figure}[H]
\center
\begin{tikzpicture}[scale=0.7] 
		\begin{axis}[ymin=0.135, ymax=0.24,xlabel={$c$}, 
		/pgf/number format/precision=5,ylabel={$\sup_{S\in \mathcal A(x)}H(x;c,S)$},smooth]
				\addplot[style=solid]		
		coordinates{
(0.000001,0.23285063476179)
(0.025,0.2190000912156)
(0.05,0.211293827638586)
(0.075,0.205067347192409)
(0.1,0.19968468082339)
(0.125,0.194877443831074)
(0.15,0.190499446806344)
(0.175,0.186459626213246)
(0.2,0.182696266457372)
(0.225,0.179165032406952)
(0.25,0.175832682888366)
(0.275,0.172673464326354)
(0.3,0.169666900025165)
(0.325,0.166796362808084)
(0.35,0.164048114128147)
(0.375,0.161410634579611)
(0.4,0.158874143810545)
(0.425,0.156430247718543)
(0.45,0.154071673661853)
(0.475,0.151792068049035)
(0.5,0.149585839102762)
(0.525,0.147448032960041)
(0.55,0.145374234797706)
(0.575,0.143360489019345)
(0.6,0.141403234168786)
(0.625,0.139499249352213)
};
		\addplot[style=dashed]		
		coordinates{
(0.000001,0.23286742672297)
(0.625,0.23286742672297)
};
\end{axis} 
	\end{tikzpicture}
		\begin{tikzpicture}[scale=0.7] 
		\begin{axis}[ymin=0.53, ymax=3.05,/pgf/number format/precision=4,xlabel={$c$}, ylabel={},smooth]
				\node at (0.4,2.6) [above] {$\bar{x}_c$};
				\node at (0.4,0.6) [above] {$\underline{x}_c$};
				\addplot[style=dotted]		
		coordinates{
(0.00000001,1.14051899)
(0.65,1.14051899)
};				
\addplot[style=solid]		
		coordinates{
(0.000001,1.13018939439675)
(0.025,0.885397281880118)
(0.05,0.830953400433187)
(0.075,0.794780864142584)
(0.1,0.767010883354394)
(0.125,0.7441928589271)
(0.15,0.7246757992576)
(0.175,0.707533400996235)
(0.2,0.692189004441446)
(0.225,0.678256652533372)
(0.25,0.665466861710429)
(0.275,0.653620133397905)
(0.3,0.642568334999336)
(0.325,0.632194281485662)
(0.35,0.622407504529138)
(0.375,0.613132790193737)
(0.4,0.604311018882349)
(0.425,0.595890690761857)
(0.45,0.58783100109037)
(0.475,0.580094981226305)
(0.5,0.572652947048029)
(0.525,0.565478812557179)
(0.55,0.558549173501193)
(0.575,0.551843878991438)
(0.6,0.545345549222892)
(0.625,0.539038681188981)
};
\addplot[style=solid]		
		coordinates{
(0.000001,1.15099157861911)
(0.025,1.52377842902634)
(0.05,1.65760639905956)
(0.075,1.76241314847323)
(0.1,1.85284650651488)
(0.125,1.93431396217022)
(0.15,2.00948919399535)
(0.175,2.07992985583354)
(0.2,2.14663249778363)
(0.225,2.2102813309119)
(0.25,2.27136868406522)
(0.275,2.33026542583137)
(0.3,2.38725780533277)
(0.325,2.44257207347516)
(0.35,2.49639469847987)
(0.375,2.54887352670632)
(0.4,2.60014204864591)
(0.425,2.65030126650336)
(0.45,2.69944887629313)
(0.475,2.74765743426721)
(0.5,2.79500807361437)
(0.525,2.84155115305079)
(0.55,2.88735093143093)
(0.575,2.93244980844827)
(0.6,2.97689113843796)
(0.625,3.0207150349188)
};
\end{axis} 
	\end{tikzpicture}
		\caption{The first graph illustrates the value for the fixed cost dividend problem as a function of $c$. 
		The second graph illustrates the corresponding dividend policy as a function of $c$.} \label{fig:asympt_small_costs}
\end{figure}
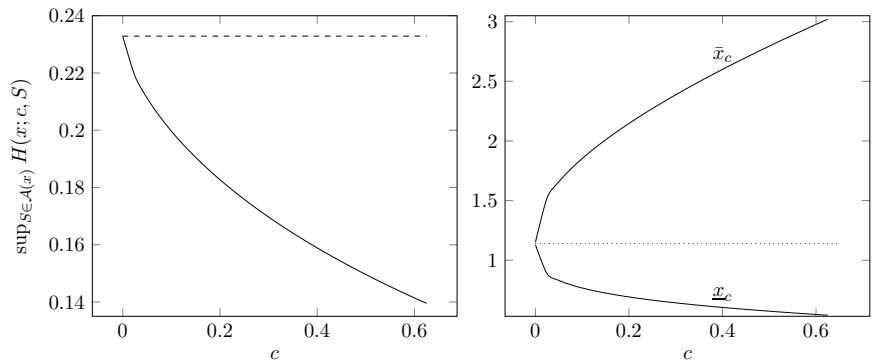

\begin{figure}[H]
\center
			\begin{tikzpicture}[scale=0.7] 
		\begin{axis}[ymin=0.23255, ymax=0.23295,/pgf/number format/precision=4,xlabel={$k$}, ylabel={$V(x_0)$},smooth]
				\addplot[style=dashed]		
		coordinates{
(0.0414599533993498,0.23286742672297)
(0.632054272689476,0.23286742672297)
};				
\addplot[style=solid]		
		coordinates{
(0.00893212311126738,0.232867370751838)
(0.0414599533993498,0.232866220306548)
(0.0893276984221259,0.232861829497781)
(0.152781418451972,0.232851057083638)
(0.192518305771818,0.232841440053187)
(0.242609835560409,0.232826171574312)
(0.300000112827751,0.232804375008439)
(0.350133850844845,0.232781583600709)
(0.400000000048371,0.23275545488356)
(0.464836614589025,0.232716343036345)
(0.504629778154618,0.232689472961664)
(0.564542335424184,0.23264492626242)
(0.632054272689476,0.232588869959545)
};
\end{axis} 
	\end{tikzpicture}
	\begin{tikzpicture}[scale=0.7] 
		\begin{axis}[ymin=1.07, ymax=1.218,xlabel={$k$}, 
		/pgf/number format/precision=5,ylabel={},smooth]
	\addplot[style=dotted]		
		coordinates{
(0.00000001,1.14051899)
(0.632,1.14051899)
};	
\addplot[style=solid]		
		coordinates{
(0.00893212311126738,1.13947859334401)
(0.0414599533993498,1.1357666405696)
(0.0893276984221259,1.13019028140108)
(0.152781418451972,1.12294120983937)
(0.192518305771818,1.11843966234291)
(0.242609835560409,1.11280641410706)
(0.300000112827751,1.10640721492708)
(0.350133850844845,1.10087160466846)
(0.400000000048371,1.09541016490648)
(0.464836614589025,1.08837759946377)
(0.504629778154618,1.08410029841758)
(0.564542335424184,1.07771635653113)
(0.632054272689476,1.07060354773639)

};
\addplot[style=solid]		
		coordinates{
(0.00893212311126738,1.14155859336817)
(0.0414599533993498,1.14542126321179)
(0.0893276984221259,1.15099129268048)
(0.152781418451972,1.15851652462865)
(0.192518305771818,1.16326590189542)
(0.242609835560409,1.16929233330686)
(0.300000112827751,1.17624855369641)
(0.350133850844845,1.18237631694033)
(0.400000000048371,1.18851234687116)
(0.464836614589025,1.19655267652033)
(0.504629778154618,1.20152233553714)
(0.564542335424184,1.2090542665279)
(0.632054272689476,1.21761233677433)
};
				\node at (0.4,1.095) [above] {$\ut{x}_k$};
				\node at (0.4,1.19) [above] {$\tilde{x}_k$};
				\end{axis} 
	\end{tikzpicture}
		\caption{
		The first graph illustrates the optimal precommitment value as a function of $k$. The second graph illustrates the corresponding dividend policy as a function of $k$.}\label{fig:asympt_precom}
\end{figure}
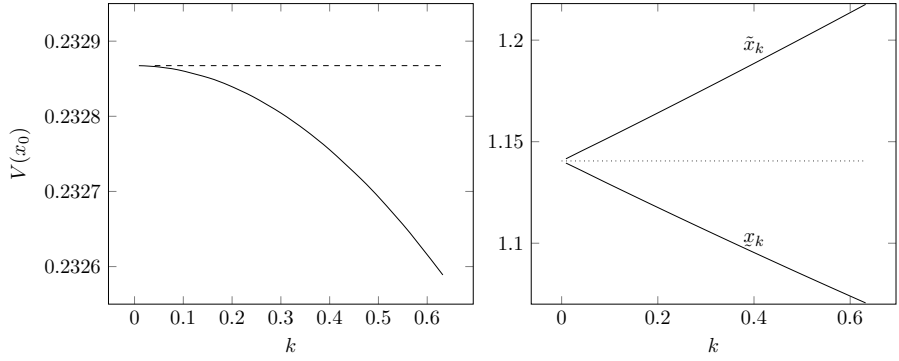

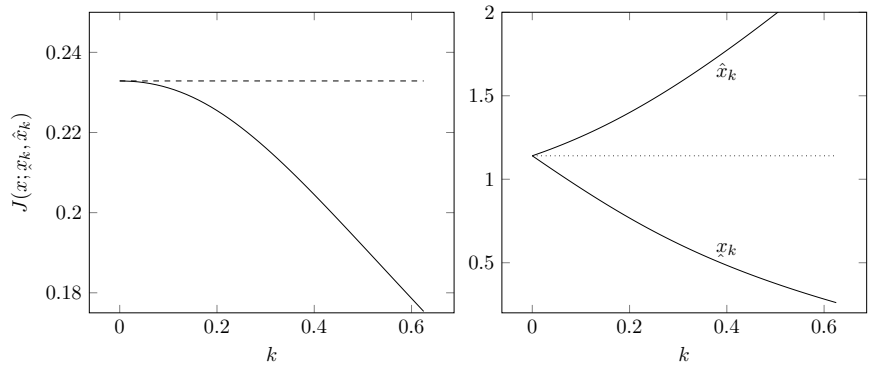
\begin{figure}[H]
\center
\begin{tikzpicture}[scale=0.7] 
		\begin{axis}[ymin=0.175, ymax=0.25,xlabel={$k$}, 
		/pgf/number format/precision=5,ylabel={$J(x;\utt x_k,\hat x_k)$},smooth]
\addplot[style=solid]		
		coordinates{
		(0.0000001,0.232867426685655)
(0.025,0.232767217573785)
(0.05,0.232454437077106)
(0.075,0.231912649361208)
(0.1,0.231128716896055)
(0.125,0.230093850027199)
(0.15,0.228804422945915)
(0.175,0.227262418099941)
(0.2,0.225475413631427)
(0.225,0.223456107001674)
(0.25,0.221221450009721)
(0.275,0.218791531519473)
(0.3,0.216188368801758)
(0.325,0.213434755329412)
(0.35,0.210553273318428)
(0.375,0.207565529648726)
(0.4,0.204491628307028)
(0.425,0.201349859430824)
(0.45,0.198156566440207)
(0.475,0.194926146475263)
(0.5,0.191671141508639)
(0.525,0.188402384203394)
(0.55,0.185129170872616)
(0.575,0.181859441906339)
(0.6,0.178599956823331)
(0.625,0.175356456373894)
};
	\addplot[style=dashed]		
		coordinates{
(0.000001,0.23286742672297)
(0.625,0.23286742672297)
};
		\end{axis} 
	\end{tikzpicture}
	\begin{tikzpicture}[scale=0.7] 
		\begin{axis}[ymin=0.2, ymax=2,xlabel={$k$}, 
		/pgf/number format/precision=5,ylabel={},smooth]
\addplot[style=solid]		
		coordinates{
(0.0000001,1.14053419739797)
(0.025,1.09075414115241)
(0.05,1.04156963966334)
(0.075,0.993134565771111)
(0.1,0.945618493539192)
(0.125,0.89918467700046)
(0.15,0.853982853804225)
(0.175,0.810142491371249)
(0.2,0.76776731530689)
(0.225,0.726931785607917)
(0.25,0.68767985778614)
(0.275,0.650025979222406)
(0.3,0.613957936298144)
(0.325,0.579440963904402)
(0.35,0.54642247838509)
(0.375,0.514836869819154)
(0.4,0.484609935918758)
(0.425,0.4556627041085)
(0.45,0.42791453299798)
(0.475,0.401285491584695)
(0.5,0.375698081725486)
(0.525,0.351078403102536)
(0.55,0.327356869608161)
(0.575,0.304468580892056)
(0.6,0.282353439891397)
(0.625,0.260956091372392)
};
\addplot[style=solid]		
		coordinates{
(0.0000001,1.14053449739952)
(0.025,1.16636790478205)
(0.05,1.19397319255645)
(0.075,1.22341227493662)
(0.1,1.254745345361)
(0.125,1.28801181299293)
(0.15,1.32322790500774)
(0.175,1.36038582068723)
(0.2,1.39945459464613)
(0.225,1.4403825425392)
(0.25,1.48310090777527)
(0.275,1.52752817370043)
(0.3,1.57357448394759)
(0.325,1.62114570747641)
(0.35,1.67014684490399)
(0.375,1.72048464331008)
(0.4,1.772069426711)
(0.425,1.82481624036799)
(0.45,1.87864544978329)
(0.475,1.93348294118785)
(0.5,1.98926005395037)
(0.525,2.04591334901887)
(0.55,2.10338428973272)
(0.575,2.16161888682592)
(0.6,2.22056734007652)
(0.625,2.28018369499548)
};
\addplot[style=dotted]		
		coordinates{
(0.0000001,1.1405189945)
(0.625000,1.1405189945)
};
				\node at (0.4,0.5) [above] {$\utt{x}_k$};
				\node at (0.4,1.55) [above] {$\hat{x}_k$};
\end{axis} 
	\end{tikzpicture}
		\caption{The first graph illustrates the equilibrium value as a function of $k$. 
		The second graph illustrates the corresponding dividend strategy as a function of $k$.}\label{fig:asympt_equili}
	\end{figure}

\appendix  
\section{Model discussion}\label{app-model}
\ref{coeff-assum:2}--\ref{coeff-assum:3} are standard assumptions that guarantee the existence of a smooth canonical solution to \eqref{ODE1}--\eqref{ODE2} and a strong (unique) solution to \eqref{state-process}. 
%
%
Adding \ref{coeff-assum:4}--\ref{coeff-assum:6} guarantees that the canonical solution $g(\cdot)$ has the properties of 
Lemma \ref{canon-sol-lemma}\ref{canon-sol-lemma:2}--\ref{canon-sol-lemma:4} which are all important ingredients in several of the proofs underlying the main results of the present paper; in particular, the strict concavity-convexity property of $g(\cdot)$, in the sense of Lemma \ref{canon-sol-lemma}\ref{canon-sol-lemma:4}, is crucial.

The fixed cost dividend problem problem \eqref{fixed-cost-prob} was in \cite[Theorem 2.1]{paulsen2007optimal} solved essentially under \ref{coeff-assum:2}--\eqref{rem-about-fixedcostres0:1}, i.e. under weaker assumptions than those of the present paper; see Remark \ref{sol-general-fixed-cost-problem} for further details. A motivation of \eqref{rem-about-fixedcostres0:1} is provided in \cite[Remark 2.1]{paulsen2007optimal}. The fixed cost dividend problem is in \cite{bai2010optimal,bai2012non} studied under further relaxations of \eqref{rem-about-fixedcostres0:1}. We leave for future research how such relaxations  
can be related to the findings of the present paper. 

If  \ref{coeff-assum:6} does not hold, i.e,. if $\mu(0) \leq 0$, then the optimal policy in (the unconstrained) problem \eqref{classical-problem} is to pay all initial surplus $x$ as a dividend immediately, see \cite[Theorem 4.3]{shreve1984optimal}. Hence, in this case the optimal solution does not violate the constraint \eqref{constraint2} (assuming that $k\leq 1$, cf. Remark \ref{equi-k-bigger-than-1}); and the constrained dividend problem is in this case not time-inconsistent and there is thus no need to investigate this case along the lines of the present paper. 

\section{Proofs}\label{app-proofs}

\begin{proof} (of Lemma \ref{canon-sol-lemma}.) 
First note that \ref{coeff-assum:2} implies that $|\mu(x)|+|\sigma(x)| \leq K(1+x)$ for all $x\geq0$ and some $K>0$. A reference for the existence of a unique canonical solution satisfying \ref{canon-sol-lemma:1} can be found on \cite[p.671]{paulsen2007optimal}. Now use $g(0)=0$, $g'(0)>0$ and \cite[Lemma 4.2 (a)]{shreve1984optimal} to see that  \ref{canon-sol-lemma:2} holds; let us however remark that we can replace \ref{coeff-assum:4} with the relaxed assumption \eqref{rem-about-fixedcostres0:1} (see page \pageref{rem-about-fixedcostres0:1}) and still \ref{canon-sol-lemma:2} holds. 

Item \ref{canon-sol-lemma:3} follows from \cite[Proposition 2.5]{bai2012optimal}; also here \ref{coeff-assum:4} can be replaced with \eqref{rem-about-fixedcostres0:1}. 
%

%
From \cite[Lemma 2.2]{paulsen2007optimal} and \ref{coeff-assum:2}--\ref{coeff-assum:4} (also here \ref{coeff-assum:4} can be replaced with \eqref{rem-about-fixedcostres0:1}) it follows that there exists a point $x_b\in [0,\infty]$ such that $g(\cdot)$ is concave on $[0,x_b)$ and convex on $[x_b,\infty)$. \ref{coeff-assum:6} implies that $x_b>0$, see \cite[Lemma 2.2]{paulsen2007optimal}. 
Clearly \ref{canon-sol-lemma:3} implies that $x_b<\infty$. It directly follows that $g''(x_b)=0$. Now, by \eqref{ODE1}, it is easy to find that 
\begin{align}
g'''(x) = 2\frac{r-\mu'(x)}{\sigma^2(x)}g'(x)-2\frac{\mu(x) + \frac{1}{2}(\sigma^2(x))'}{\sigma^2(x)}g''(x).
\end{align}
Hence, relying on \ref{coeff-assum:2}--\ref{coeff-assum:4}
%
we apply \cite[Lemma 4.1]{shreve1984optimal} to $g'(\cdot)$ and obtain \ref{canon-sol-lemma:4}; we remark that a similar argument is made in the proof of \cite[Lemma 4.2]{shreve1984optimal}.%
%
\end{proof}

\begin{proof} (of Proposition \ref{expl-JR}.) The claims can be proved using the usual arguments involving the strong Markov property, It\^{o}'s formula, and that $J(x;\underline{x},\bar{x})$ and $R(x;\underline{x},\bar{x})$ satisfy the ODE in \eqref{ODE1} for $0<x<\bar{x}$ and the boundary conditions 
\begin{align}
&J(0;\underline{x},\bar{x})=0, \enskip J(x;\underline{x},\bar{x})=J(\underline{x};\underline{x},\bar{x})+x-\underline{x}
\mbox{ for } x\geq \bar{x},\\
&R(0;\underline{x},\bar{x})=0, \enskip R(x;\underline{x},\bar{x})=R(\underline{x};\underline{x},\bar{x})+1 \mbox{ for } x\geq \bar{x}.
\end{align}
\end{proof}

\begin{proof} (of Lemma \ref{propR}.) 
The claims can be verified using \eqref{expl-JR:R}, Lemma \ref{canon-sol-lemma} (which e.g. implies that $\lim_{x\rightarrow \infty}g(x)=\infty$) and $\bar{x}>\underline{x}\geq0$. 
\end{proof}

\begin{proof} (of Proposition \ref{fix-cost-rem:properties}.) 
Use \eqref{H-function} to see that $H(\bar{x}_c;c,\underline{x}_c,\bar{x}_c)>0$ and hence $\bar{x}_c - \underline{x}_c-c>0$. 
Moreover, 
$H(\bar{x}_c;c,\underline{x}_c,\bar{x}_c)-H(\underline{x}_c;c,\underline{x}_c,\bar{x}_c) = \bar{x}_c - \underline{x}_c-c>0$. Using also that $\underline{x}_c\geq0$ we find that \eqref{sec:pre-commitment-solution:property3} holds. 

Suppose a function $f:\mathbb{R}\rightarrow \mathbb{R}$ satisfies 

\begin{itemize}
\item $f'(\underline{y})=f'(\bar{y})=1$ for some $0\leq\underline{y}<\bar{y}$, and
\item $f''(y)<0$ for $y\in(0,y^*)$ and $f''(y)>0$ for $y\in(y^*,\infty)$ for some $y^*>0$, 
\end{itemize}
then it directly follows that $0\leq\underline{y}<y^*<\bar{y}$. Note that Lemma \ref{canon-sol-lemma} implies that $g(\cdot)$ satisfies the second item above with $y^*=x^*$ (recall that $x^*:=x_b$). It follows, cf. \eqref{H-function}, that $H(\cdot;c,\underline{x}_c,\bar{x}_c)$ satisfies the second item above; moreover, $H(\cdot;c,\underline{x}_c,\bar{x}_c)$ satisfies the first item above (with $\underline{x}_c=\underline{y}$ and $\bar{x}_c=\bar{y}$)  within case \ref{fix-cost-rem:1} (of Proposition \ref{fix-cost-rem}).   
%
Hence, \eqref{sec:pre-commitment-solution:property0} holds in case \ref{fix-cost-rem:1}. It can be similarly shown that \eqref{sec:pre-commitment-solution:property0} holds in  case \ref{fix-cost-rem:2}.

Consider case \ref{fix-cost-rem:2}. From \eqref{sec:pre-commitment-solution:property0} and \eqref{sec:pre-commitment-solution:property3} we know that $\bar{x}_c>x^*,c$. 
%
%
Hence, using also \eqref{H-function} and \eqref{fixed-cost-smoothfit2}, we see that $\bar{x}_c$ is the unique solution to the equation
\begin{align}   
A(\bar{x},c):= g'(\bar{x}) (\bar{x}-c)-g(\bar{x}) = 0, \enskip \bar{x}>x^*,c. \label{eq-systfixed-proof:0}
\end{align}
Note that 
$A_{\bar{x}}(\bar{x},c) = g''(\bar{x}) (\bar{x}-c)>0$ and 
$A_{c}(\bar{x},c) = -g'(\bar{x})<0$ for all $\bar{x}>x^*,c$ (Lemma \ref{canon-sol-lemma}). 
Hence, by the implicit function theorem, $\bar{x}_c=\bar{x}_c(c)$ for a function $\bar{x}_c(\cdot)$ satisfying 
\[\bar{x}_c'(c) = - \frac{A_{c}(\bar{x}_c,c)}{A_{\bar{x}}(\bar{x}_c,c)} = \frac{g'(\bar{x}_c)}{g''(\bar{x}_c) (\bar{x}_c-c)}>0.\] 
It follows that $\bar{x}_c$ is (strictly) increasing and continous in $c$ within case \ref{fix-cost-rem:2}. In case \ref{fix-cost-rem:1} it similarly holds that 
$(\underline{x}_c,\bar{x}_c)$ is 
the unique solution to the equation system
\begin{align}
\begin{split}   
B(\underline{x},\bar{x},c)&:= g'(\underline{x}) (\bar{x}-\underline{x}-c)-g(\bar{x})+ g(\underline{x}) = 0 \\
C(\underline{x},\bar{x},c)&:= g'(\bar{x}) (\bar{x}-\underline{x}-c)-g(\bar{x})+ g(\underline{x}) = 0 \\
&\mbox{for $\bar{x}>x^*,c$ and $x^*>\underline{x}>0$.} \label{eq-systfixed-proof}
\end{split}
\end{align}
Note that, for all $\bar{x}>x^*,c$ and $x^*>\underline{x}>0$,  
\begin{itemize}
\item $B_{\underline{x}}(\underline{x},\bar{x},c) = g''(\underline{x}) (\bar{x}-\underline{x}-c)<0$, 
\item $B_{\bar{x}}(\underline{x},\bar{x},c)= C_{\underline{x}}(\underline{x},\bar{x},c) = g'(\underline{x})-g'(\bar{x})$,
\item $B_c(\underline{x},\bar{x},c) = -g'(\underline{x})<0$,
\item $C_{\bar{x}}(\underline{x},\bar{x},c) = g''(\bar{x}) (\bar{x}-\underline{x}-c)>0$,
\item $C_c(\underline{x},\bar{x},c) = -g'(\bar{x})<0$.
\end{itemize}
The Jacobian matrix of $(\underline{x},\bar{x})\mapsto (B(\underline{x},\bar{x},c),C(\underline{x},\bar{x},c))^{T}$ (where $^T$ denotes transpose), cf. \eqref{eq-systfixed-proof}, is
\begin{align}
 \begin{pmatrix}
  g''(\underline{x})(\bar{x}-\underline{x}-c) & g'(\underline{x})-g'(\bar{x}) \\
	g'(\underline{x})-g'(\bar{x}) & g''(\bar{x})(\bar{x}-\underline{x}-c)
  \end{pmatrix},
\end{align}
and hence its determinant is 
\begin{align}
g''(\bar{x})g''(\underline{x}) (\bar{x}-\underline{x}-c)^2-(g'(\underline{x})-g'(\bar{x}))^2<0,
\end{align}
where we used that $g''(\underline{x})<0$ and $g''(\bar{x})>0$ (to see this use $\underline{x}>x^*>\bar{x}$ and Lemma \ref{canon-sol-lemma}). The Jacobian matrix is therefore invertible. Note that $g'(\underline{x}_c)=g'(\bar{x}_c)$, cf. \eqref{eq-systfixed-proof}, and use the implicit function theorem to see that $(\underline{x}_c,\bar{x}_c)=(\underline{x}_c,\bar{x}_c)(c)$ for a function $(\underline{x}_c,\bar{x}_c)(\cdot)$ satisfying 
\[
\left(\frac{d}{dc}(\underline{x}_c,\bar{x}_c)(c) \right)^T = 
- 
\begin{pmatrix}
  g''(\underline{x}_c)(\bar{x}_c-\underline{x}_c-c) & 0 \\
	0 & g''(\bar{x}_c)(\bar{x}_c-\underline{x}_c-c)
\end{pmatrix}^{-1}
\begin{pmatrix}
   -g'(\underline{x}_c)\\
	 - g'(\bar{x}_c)
\end{pmatrix}
\]
\[
\enskip \enskip \enskip  \enskip \enskip \enskip \enskip \enskip \enskip \enskip \enskip \enskip 
= \frac{1}{g''(\bar{x}_c)g''(\underline{x}_c) (\bar{x}_c-\underline{x}_c-c)^2}\begin{pmatrix}
   g''(\bar{x}_c)(\bar{x}_c-\underline{x}_c-c)g'(\underline{x}_c)\\
	 g''(\underline{x}_c)(\bar{x}_c-\underline{x}_c-c)g'(\bar{x}_c)
\end{pmatrix}
\]
\[
=\begin{pmatrix}
   \frac{g'(\underline{x}_c)}{g''(\underline{x}_c)(\underline{x}_c-\bar{x}_c-c)}\\
	 \frac{g'(\bar{x}_c)}{g''(\bar{x}_c)(\underline{x}_c-\bar{x}_c-c)}
\end{pmatrix}.
\enskip \enskip \enskip  \enskip \enskip \enskip \enskip \enskip \enskip \enskip \enskip \enskip
\enskip \enskip \enskip  \enskip \enskip \enskip \enskip \enskip  
\]
Hence, 
$\underline{x}_c$ is (strictly) decreasing and continous in $c$ 
while 
$\bar{x}_c$ is (strictly) increasing and continous in $c$ 
within case \ref{fix-cost-rem:1}.  Using the observations above it is easy to see that if the following statements,
\begin{align}
\begin{split}  
&\mbox{(1) there exists a $\bar{c}$ which is such that if  $c<\bar{c}$ then we are in case \ref{fix-cost-rem:1} and if}\\
&\mbox{$c\geq\bar{c}$, then we are in case \ref{fix-cost-rem:2}, and (2) $(\underline{x}_c,\bar{x}_c)(c)\rightarrow (0,(\bar{x}_c)(\bar{c}))$ as $c\nearrow\bar{c}$,}\label{fix-cost-rem:properties:h1}
\end{split}
\end{align} 
hold, then \eqref{sec:pre-commitment-solution:property1} holds, and so does \eqref{sec:pre-commitment-solution:property3.5}. Recall that $(\underline{x}_c,\bar{x}_c)$ is the unique solution to \eqref{eq-systfixed-proof} in case \ref{fix-cost-rem:1}; but this is equivalent to 
\begin{align} 
g'(\underline{x}_c) = g'(\bar{x}_c) = \frac{g(\bar{x}_c)- g(\underline{x}_c)}{\bar{x}_c-\underline{x}_c-c}
\left(>\frac{g(\bar{x}_c)- g(\underline{x}_c)}{\bar{x}_c-\underline{x}_c}\right). \label{fix-cost-rem:properties:h2}
\end{align} 
Similarly, in case \ref{fix-cost-rem:2} note that \eqref{eq-systfixed-proof:0} is equivalent to
\begin{align} 
g'(\underline{x}_c)= \frac{g(\bar{x}_c)}{\bar{x}_c-c} 
\left(=\frac{g(\bar{x}_c)- g(0)}{\bar{x}_c-0-c}>\frac{g(\bar{x}_c)- g(0)}{\bar{x}_c-0}\right). \label{fix-cost-rem:properties:h3}
\end{align} 
Note that \eqref{fix-cost-rem:properties:h2} means that, for any $c$, the derivatives $g'(\underline{x}_c)=g'(\bar{x}_c)$ strictly dominate the slope of a line between the points $(\underline{x}_c,g(\underline{x}_c))$ and $(\bar{x}_c,g(\bar{x}_c))$ and that the difference between these derivatives and the slope is strictly decreasing in $c$. A similar observation can be made for \eqref{fix-cost-rem:properties:h3}. Using the observations above and recalling that $g(\cdot)$ is strictly concave on $(0,x^*)$ and strictly convex on $(x^*,\infty)$ it is easy to see that \eqref{fix-cost-rem:properties:h1} holds
, which thus implies that \eqref{sec:pre-commitment-solution:property1} and \eqref{sec:pre-commitment-solution:property3.5} hold; and moreover that
\[g'(\underline{x}_c) = g'(\bar{x}_c) = \frac{g(\bar{x}_c)- g(\underline{x}_c)}{\bar{x}_c-\underline{x}_c-c} \rightarrow g'(x^*) 
\enskip \mbox{as $c\searrow 0$},\]
and 
\[\underline{x}_c \nearrow x^* \enskip \mbox{and} \enskip \bar{x}_c \searrow x^* \enskip \mbox{as $c\searrow 0$}.\]
It thus also follows that \eqref{sec:pre-commitment-solution:property2} and \eqref{sec:pre-commitment-solution:property4} hold (to see this use also e.g. \eqref{uncon-problembU} and \eqref{H-function}). 
%
%
%
%
%
%
\end{proof}

\begin{proof} (of Lemma \ref{lem2}.) The claims can be verified using Lemma \ref{propR}\ref{propR:part1} and properties of the function $c\mapsto (\underline{x}_{c},\bar{x}_{c})$ corresponding to Proposition \ref{fix-cost-rem:properties}. 
%
\end{proof}

\begin{proof} (of Lemma \ref{lemma-equi}.) Proposition \ref{expl-JR} and the fact that the functions in \eqref{eq-eq2:constraint}--\eqref{eq-eq2:smooth-fit} are defined only for ${\underline{x}}>{\bar{x}}\geq 0$ are used throughout the proof.

Let us prove \ref{lemma-equi:part1}: Suppose $k=1$. It is easy to see that \eqref{eq-eq2:constraint} holds if and only if ${\underline{x}}=0$. Using this and \eqref{eq-eq2:smooth-fit} it is easy to see that the claim holds if and only if the equation 
\begin{align} 
A(\bar{x}):=g'(\bar{x})\bar{x}-g(\bar{x}) =0
\end{align}
has exactly one solution in $(0,\infty)$ and this solution is strictly larger than $x^*$. 
To see that this the case it suffices to note that: 
\begin{itemize} 
\item $\lim_{\bar{x} \searrow 0}A(\bar{x})=A(0)=0$ and $\lim_{\bar{x} \rightarrow \infty}A(\bar{x})=\infty$.
\item $A'(\bar{x})= g''(\bar{x})\bar{x}$ which, by Lemma \ref{canon-sol-lemma}, means that $A'(\bar{x})<0$ for $\bar{x} \in (0,x^*)$ and $A'(\bar{x})>0$ for $\bar{x} >x^*$.
\end{itemize}  
Now suppose $k<1$ (recall that $k>0$, see \eqref{constraint2}). It follows from \eqref{eq-eq2:constraint} that ${\underline{x}}>0$. 
Recalling that $\bar{x}>\underline{x}\geq 0$ by definition it is easy to see that \eqref{eq-eq2:smooth-fit} holds if and only if 
\begin{align}
g'(\bar{x}) = \frac{g(\bar{x})-g(\underline{x})}{{\bar{x}}-{\underline{x}}}. \label{lemma-equi:pfeq1}
\end{align}
Now recall from Lemma \ref{canon-sol-lemma} that 
$g''(x)<0$ if  $x\in (0,x^*)$ and $g''(x)>0$ if  $x\in (x^*,\infty)$, and $\lim_{x\rightarrow\infty}g'(x)=\infty$, from which it is easy to see 
that:
\begin{itemize} 

\item If $\underline{x} \geq x^*$ then no $\bar{x}>\underline{x}$ such that \eqref{lemma-equi:pfeq1} holds exists.

\item If $\underline{x} < x^*$ then a unique $\bar{x}>\underline{x}$ such that \eqref{lemma-equi:pfeq1} holds exists, and $\bar{x}>x^*$.

\item There exists a continous strictly decreasing function 
\begin{align} 
\bar{x}(\cdot) \label{new-func-july}
\end{align}
such that \eqref{lemma-equi:pfeq1} holds if and only if $\bar{x}=\bar{x}(\underline{x})$ for any a fixed $\underline{x} \in (0,x^*)$, where $\bar{x}(\underline{x})>x^*$. Moreover, sending $\underline{x}\nearrow x^*$ implies that $\bar{x}(\underline{x})\searrow x^*$.

\end{itemize}
%
 
From the observations above and properties of the function $\bar{x}\mapsto R(\bar{x};\underline{x},\bar{x})$ (Lemma \ref{propR}) it follows that the following schedule gives a unique solution to \eqref{eq-eq2:constraint}--\eqref{eq-eq2:smooth-fit} in the case $k<1$:
\begin{itemize} 
\item Pick a $\underline{x}_1 \in (0,x^*)$. 

\item Determine $\bar{x}_1=\bar{x}(\underline{x}_1)$ by verifying \eqref{lemma-equi:pfeq1} (which implies that \eqref{eq-eq2:smooth-fit} holds).

\item If $R\left(\bar{x}_1;\underline{x}_1,\bar{x}_1\right) > \frac{1}{k}$ then we know that $\underline{x}_1$ is too high and  $\bar{x}_1$ is too low (cf. that $\bar{x}(\cdot)$ is decreasing) to be a solution to \eqref{eq-eq2:constraint}, and we set $\underline{x}_1=\underline{x}_1^h$ and choose a new $\underline{x}_2<\underline{x}_1^h$. 
Analogously, if $R\left(\bar{x}_1;\underline{x}_1,\bar{x}_1\right) < \frac{1}{k}$ then we know that $\underline{x}_1$ is too low and we set $\underline{x}_1=\underline{x}_1^l$ and choose a new $\underline{x}_2>\underline{x}_1^l$. 

\item Iterate the steps above always choosing  (using e.g. the bisection method) $\underline{x}_{n+1}$ smaller than all $\underline{x}_i^h,i\leq n$ and larger than all $\underline{x}_i^l,i\leq n$ until you find an $\underline{x}$ such that, with $\bar{x}=\bar{x}(\underline{x})$, it holds that 
$R\left(\bar{x};\underline{x},\bar{x}\right) = \frac{1}{k}$ (which by Lemma \ref{propR} is possible). 
\end{itemize}
Note that this schedule can easily be modified in order to find the solution to \eqref{eq-eq2:constraint}--\eqref{eq-eq2:smooth-fit} in the case $k=1$, using that in this case \eqref{eq-eq2:constraint} holds if and only if $\underline{x}=0$.  

Let us now prove \ref{lemma-equi:part2}: It is easy to verify that \eqref{eq-eq2:constraint} holds if and only if $(\underline{x},\bar{x})$ is a solution to the equation 
\begin{align}
g(\underline{x})
= (1-k)g(\bar{x}).\label{lemma-equi:pfeq2} 
\end{align}
Inserting \eqref{lemma-equi:pfeq2}  into \eqref{lemma-equi:pfeq1} yields \eqref{lemma-equi2:eq1} and inserting \eqref{lemma-equi2:eq1} into \eqref{lemma-equi:pfeq2}  yields 
\begin{align} 
g\left(\bar{x}-k\frac{g(\bar{x})}{g'(\bar{x})}\right) - (1-k)g(\bar{x})=0\label{lemma-equi2:eq2}.
\end{align}
The claim can now be verified using \eqref{expl-JR:J}. Item \ref{lemma-equi:part3} follows from \eqref{expl-JR:J}.
%
 \end{proof}

\begin{proof} (of Theorem \ref{sol-conv-equi}.)
Recall that $(\utt x_k,\hat x_k)$ solves \eqref{eq-eq2:constraint}--\eqref{eq-eq2:smooth-fit}. 
Hence, using Lemma \ref{lemma-equi} we obtain \eqref{sol-conv-equi:1}. 
To see that \eqref{sol-conv-equi:3} and \eqref{sol-conv-equi:4} hold recall the properties of the function \eqref{new-func-july} and adapt the arguments after \eqref{new-func-july} in the obvious way, recalling also \eqref{lemma-equi:pfeq1} and properties of the function $\bar{x}\mapsto R(\bar{x};\underline{x},\bar{x})$ (Lemma \ref{propR}). 
Item \eqref{sol-conv-equi:2} is proved similarly. 
Item \eqref{sol-conv-equi:6} follows from Theorem \ref{sol-equi} and Lemma \ref{lemma-equi}.
\end{proof}

\bibliographystyle{abbrv}
\bibliography{time-incon_stopping}

\end{document}